%% file: CHM,_GenCycle.tex
\documentclass[11pt]{amsart}

\usepackage{amssymb,amscd,amsthm}
\usepackage[all,line,arc,curve,color,frame,pdf]{xy}
\usepackage{tikz} 
\usepackage{tikz-cd}
\usetikzlibrary{calc,patterns,positioning}
\usepackage{hyperref}
\usepackage[shortlabels]{enumitem}
\usepackage{enumitem}

\newcommand\ignore[1]{}

\usepackage{pifont}
\newcommand{\cmark}{\ding{51}}
\newcommand{\xmark}{\ding{55}}

\newcommand\CC{{\mathbb{C}}}

\newcommand\RR{{\mathbb{R}}}

\newcommand\ZZ{{\mathbb{Z}}}

\def\C{{\mathbb C}}

\def\P{{\mathbb P}}

\def\operatorname#1{\mathop{\rm #1}\nolimits}

\def\deg{\operatorname{deg}}

\def\det{\operatorname{det}}

\newcommand{\pb}{\ar@{}[dr]|{\text{\pigpenfont J}}}

\makeatletter
\newcommand{\xleftrightarrow}[2][]{\ext@arrow 3359\leftrightarrowfill@{#1}{#2}}
\makeatother
\newcommand{\xdasharrow}[2][->]{
\tikz[baseline=-\the\dimexpr\fontdimen22\textfont2\relax]{
\node[anchor=south,font=\scriptsize, inner ysep=1.5pt,outer xsep=2.2pt](x){#2};
\draw[shorten <=3.4pt,shorten >=3.4pt,dashed,#1](x.south west)--(x.south east);
}}

\numberwithin{equation}{section}
\theoremstyle{plain}
\newtheorem{thm}{Theorem}[section]
\newtheorem{cor}[thm]{Corollary}
\newtheorem{lemma}[thm]{Lemma}
\newtheorem{prop}[thm]{Proposition}

\theoremstyle{remark}
\newtheorem{rem}[thm]{Remark}
\newtheorem{ex}[thm]{Example}

\theoremstyle{definition}
\newtheorem{defi}[thm]{Definition}

\DeclareMathOperator\initial{in_{<}}

\begin{document}
\title{Sullivant-Talaska ideal of the cyclic Gaussian Graphical Model}

\author[Conner]{Austin Conner}
\address{University of Konstanz, Germany, Fachbereich Mathematik und Statistik, Fach D 197
D-78457 Konstanz, Germany}
\curraddr{Department of Mathematics, Harvard University, 1 Oxford St, Cambridge, MA, 02138}
\email{aconner@math.harvard.edu}

\author[Han]{Kangjin Han}
\address{ 
School of Undergraduate Studies, Daegu-Gyeongbuk Institute of Science \& Technology (DGIST),
Daegu 42988, Republic of Korea}
\email{kjhan@dgist.ac.kr}

\author[Micha{\l}ek]{Mateusz Micha{\l}ek}
\address{
University of Konstanz, Germany, Fachbereich Mathematik und Statistik, Fach D 197
D-78457 Konstanz, Germany
}
\email{mateusz.michalek@uni-konstanz.de}

\thanks{A.C.~supported by NSF grant 2002149 and DFG grant 467575307, K.H.~supported by a National Research Foundation of Korea (NRF) grant (MSIT no. 2021R1F1A104818611) and DGIST Global Visiting Research Program, M.M.~supported by the DFG grant 467575307.}

\begin{abstract}
In this paper, we settle a conjecture due to Sturmfels and Uhler concerning generation of the prime ideal of the variety associated to the Gaussian graphical model of any cycle graph. Our methods are general and applicable to a large class of ideals with radical initial ideals. 
\end{abstract}



\maketitle
\section{Introduction}
The main inspiration for our work is the groundbreaking article by Sturmfels and Uhler \cite{StUh}. It sets-up foundations for algebraic study of linear, and in particular graphical, Gaussian models. In short, to a linear subspace $L$ of symmetric $n\times n$ matrices one associates an algebraic variety $L^{-1}$, which is the Zariski closure of inverses of (invertible) matrices in $L$. 

To stress how important this work is, let us note that when $L$ is a space of diagonal matrices, there is a representable matroid associated to it, or in fact two closely related matroids \cite[\S3.3]{Huh}, \cite[\S4.1]{DMS2}. Many invariants of $L^{-1}$, like the degree, are related to basic invariants of the matroid, like beta invariant or constant coefficient of the characteristic polynomial \cite[\S3]{StUh}. This was later fundamental in the work of June Huh, who made further connections relating the multidegree of the graph of the inversion map $L\rightarrow L^{-1}$ with the coefficients of the reduced characteristic polynomial, confirming the long-standing open conjecture of Read \cite{Huh}.

Still, from the point of view of statistics, the most interesting spaces $L$ are not contained in the diagonal matrices, but rather contain the whole space of diagonal matrices. For such general $L$, inspired by the results of Huh, the formulas for the degree of $L^{-1}$ were obtained in \cite{MMMSV}, basing on results from \cite{MMW}.
However, special $L$ corresponding to Gaussian graphical models play a central role in algebraic statistics, while still many results about them are mostly conjectural.
 
Given a graph $G$ on $n$ vertices, one defines a linear space $L_G$ of symmetric matrices --- details are given in \S\ref{sec:graphical}. Our main result is the following theorem, confirming \cite[Conjectured equation $(23)$]{StUh}.
\begin{thm}\label{thm:main}
Let $C_n$ be the cycle graph on $n$ vertices. The special $3\times 3$ minors generating its Sullivant-Talaska ideal generate also $I(C_n)$, the prime ideal of the cyclic Gaussian graphical model. Hence, the two ideals are equal.
\end{thm}
A detailed description of the special $3\times 3$ minors and the Sullivant-Talaska ideal $I_n$ is given in \S\ref{sec:cyclic}.

In fact, we prove a stronger claim: these special $3\times 3$ minors form a square-free Gr\"obner basis for the ideal of the cyclic Gaussian graphical model under an appropriately chosen term order. One of the key steps is a simple but general Lemma \ref{lem:general}. It says that for any irreducible projective variety $X$, if we have a set of polynomials in $I(X)$ with square-free leading monomials, we just have to check if the leading monomials define an equidimensional variety with the same dimension and degree as $X$, to conclude that this is a Gr\"obner basis. Another recent application of this lemma, coauthored by two authors of this article, can be found in \cite{CMSS}.

Let $I_n$ be the Sullivant-Talaska ideal for the $n$-cycle $C_n$. 
Our proof has following steps.
\begin{enumerate}
  \item In \S\ref{sec:cyclic}, we pick a term order so that the generators of
    $I_n$ have square-free leading monomials, and we study the 
  radical ideal $J$ generated by the leading
  monomials of the generators.
  The associated variety consists of a union of coordinate subspaces,
  and we establish a bijection between such subspaces and 
  maximal subsets of secants in a regular $n$-gon not containing explicit size
  three subsets. We name such a configuration a Maximal Set Avoiding Forbidden
  Triples, or Msaft. 
 \item 
  The cardinality of an Msaft gives the dimension of the corresponding
  subspace, and in \S\ref{sec:eqdim} we show every Msaft has the same
  cardinality $2n$, i.e.~the initial ideal $J$ is equidimensional of dimension
  $2n$, the same as $I(C_n)$.
\item Next, we compute the degree of $J$ by counting the Msafts. To achieve this,
  we observe in \S\ref{subsec:countMS} that one can put a directed graph
  structure on the set of all secants so that each Msaft admits a unique cyclical walk
  in this graph through the secants belonging to it, and we obtain a simpler
  characterization of Msafts directly in terms of such walks. Then, we
  translate the problem of counting such walks as counting pairs of vertex
  disjoint pairs of paths in a certain finite lattice. Using the
  Lindstr\"om-Gessel-Viennot lemma and the reflection principle, we obtain an
  explicit formula using sums and binomial coefficients. 
\item In \S\ref{sec:binoms}, we perform the needed summations to find a closed
  formula for the degree of $J$ which matches the degree of $I(C_n)$ recently
  computed in \cite{MRV}. Applying our general Lemma \ref{lem:general} then
  establishes that the cubic generators of $I_n$ form a Gr\"obner basis for
  $I(C_n)$, which completes the proof.
\end{enumerate}


\section{Gaussian Graphical Models in a Nutshell}\label{sec:graphical}
Linear concentration models are special Gaussian models introduced by Anderson over half a century ago \cite{Anderson}. In purely mathematical terms such a model is represented by a linear space $L$ of $n\times n$ symmetric matrices. Each positive definite matrix $K\in L$ provides a probability distribution on $\RR^{n}$ proportional to $e^{-{\bf x}^T K{\bf x}}$. The variety $X_L$ associated to the model is the Zariski closure of the locus of inverses of all invertible matrices from $L$ in the space of symmetric $n\times n$ complex matrices. As it is invariant under scaling, we may view it also as a projective variety. To be precise, in statistics, the model would be a semi-algebraic locus given by inverses of positive definite matrices in $L$ (which are also positive definite). 

From the point of view of algebraic statistic the most interesting and well-studied family of linear concentration models are Gaussian graphical models. Consider a simple graph $G$ on $n$ vertices. To such a graph one associates a coordinate subspace $L_G$ of $n\times n$ symmetric matrices defined by:
\[x_{ij}=0 \text{ when } i\neq j \text{ and }\{i,j\}\text{ is not an edge of }G,\]
where the entries of the matrix are $(x_{ij})$. For more information we refer to \cite[\S8.3]{Seth}, \cite{StUh} and \cite{DSS}.

\begin{ex}\label{eg_C5}
Let $G$ be the cycle graph $C_5$ on five vertices. Then, $L_G$ is the following $10$-dimensional subspace in $S^2(\C^5)$ 
\[
K=\begin{pmatrix}
x_{11}&x_{12}&0&0&x_{15}\\
x_{12}&x_{22}&x_{23}&0&0\\
0&x_{23}&x_{33}&x_{34}&0\\
0&0&x_{34}&x_{44}&x_{45}\\
x_{15}&0&0&x_{45}&x_{55}\\
\end{pmatrix}. \]
The prime ideal $I(C_5)$ defining $X_{L_{C_5}}$ in $\C[\sigma_{11},\sigma_{12},\ldots,\sigma_{55}]$ can be computed by eliminating the variables $x_{ij}$ from the system of equations
\[ \Sigma \cdot K=I_5~,\]
where $\Sigma=(\sigma_{ij})$ is a symmetric $5\times5$ matrices of variables $\sigma_{ij}$ and $I_5$ is the identity matrix. We note that this kind of elimination computation is usually infeasible for large $n$.
\end{ex}

In this article we focus on the case when $G$ is $C_n$, the $n$-cycle graph. Our main theorem describes the prime ideal $I(C_n)$ of the variety $X_{L_{C_n}}$ in $\P(S^2(\C^n))$. As we prove, it is generated by \textit{cubics} corresponding to special minors of symmetric matrices. To our knowledge, so far a major breakthrough for defining equations was only achieved in the case when $G$ was a \textit{block graph} (i.e. $1$-clique sum of complete graphs). In this case a complete description of the ideal of $X_G$ was recently provided by Misra and Sullivant \cite{MS} confirming another conjecture by Sturmfels and Uhler. However, in those cases the ideal is generated in degrees one and two and moreover is toric. Thus, the methods and approach we present are very different.  

\section{Gr\"obner basis}\label{sec:GB}
\subsection{General tools}\label{subsec:gentools}
\begin{lemma}\label{lem:general}
Let $I$ be a homogeneous prime ideal in $R=\CC[x_0,\dots,x_n]$ and $g_1,\dots,g_k$ be elements in $I$. Assume that for some term order $<$ 
\begin{enumerate}
\item $\initial(g_i)$ is square-free for each $i$,
\item $J:=\langle \initial(g_1),\dots, \initial(g_k)\rangle$ is equidimensional,
\item $\dim J=\dim I$ and $\deg(J)=\deg(I)$.
\end{enumerate}
Then, $\{g_1,\dots,g_k\}$ form a Gr{\"o}bner basis for $I$ with respect to $<$.
\end{lemma}
\begin{proof}
First, note that $J$ is a radical monomial ideal by (1). By (2) and (3),  we know that the zero set $V(J)$ is a union of $deg(I)$-many coordinate subspaces of dimension equal to $\dim I$. We claim that $J=\initial(I)$. 

Obviously, we have $J\subset \initial(I)$. To obtain the other containment, let us suppose $J\not\supset \initial(I)$. Then, there exists a monomial $f\in \initial(I)\setminus J$. Since $J\subsetneq J+(f) \subseteq \initial(I)$, by Nullstellensatz, we obtain $$V(\initial(I))\subseteq V(J+(f)) \subsetneq V(J)~$$
(in particular, $V(J+(f))$ is a strictly smaller set than $V(J)$). Since $V(J)$ is a union of linear subspaces of the same dimension, $\dim (J+(f))< \dim J$ or $\deg(J+(f))<\deg(J)$. But, it is impossible, as both $\initial(I)$ and $J$ have the same dimension as $\dim I$ and the same degree as $\deg(I)$.
\end{proof}

\begin{rem} We note that one also has a `partial converse' of Lemma \ref{lem:general} in the following sense: Suppose that $\{g_1,\dots,g_k\}$ is a Gr{\"o}bner basis for $I$ with respect to a term order $<$.

First of all, $J:= \initial(I)$ has the same dimension and degree as $I$ (i.e. condition (3) in Lemma \ref{lem:general} is satisfied). Once we know that $J$ is a radical ideal (i.e. every $\initial(g_i)$ is square-free, which is condition (1)), then we can also see that $J$ is equidimensional (condition (2)) by \cite[Corollary 2.1.8]{BCRV}. Thus, for any term order $<$ providing a radical initial ideal, condition (2) and (3) in Lemma \ref{lem:general} are equivalent to the fact that $\{g_1,\dots,g_k\}$ is a Gr{\"o}bner basis for $I$ with respect to the order $<$.
\end{rem}

\subsection{Cyclic model}\label{sec:cyclic}
Let $R:=\CC[\sigma_{i,j}]_{1\leq i\leq j\leq n}$ be the polynomial ring identified with the ring of polynomials on the space of symmetric $n\times n$ matrices. Let $\Sigma=(\sigma_{i,j})$ be the generic symmetric matrix, filled with the indeterminates. 

The \textit{Sullivant-Talaska ideal} $I_n$ was defined in \cite{StUh}, based on the earlier work by Draisma, Sullivant and Talaska \cite{DST}. Below we present its construction for the $n$-cycle $C_n$. The ideal $I_n \subset R$ is generated by all the $3\times 3$ subminors obtained as follows. We fix a cyclic notation for the set $[n]=\{1,\dots, n\}$. Consider any interval $[a,b]\subset \{1,\dots, n\}$. Please note that it may be of the form $\{a=n-1,n,1,2,b=3\}$. Consider the complimentary interval intersecting the given one only in $a,b$, that is $[b,a]:=([n]\setminus [a,b])\cup \{a,b\}$. For any choice of $a,b$ we take all the $3\times 3$ minors of the $[a,b]\times [b,a]$ submatrix of $\Sigma$, that is indices of rows belong to $[a,b]$ and of columns to $[b,a]$. 

Sturmfels and Uhler \cite{StUh} conjectured that $I_n$ is the (radical) ideal defining the irreducible variety $X_{L_{C_n}}$ associated to the cyclic Gaussian graphical model (note that for $n=3$ the conjecture holds trivially, since $X_{L_{C_3}}$ is the whole space and both ideals are zero).

In order to better understand $I_n$ for any $n\ge4$, we fix the following term order on $R$. 
\begin{defi}\label{def:order}
We consider any lexicographic order on  $R$, where $\sigma_{i,j}>\sigma_{i',j'}$, whenever $|i-j|<|i'-j'|$ and otherwise we order the variables arbitrary. Less formally: the closer to the diagonal, the bigger the variable. 
\end{defi}

\begin{ex}\label{eg_C5_eqn}
Let $G$ be the cycle graph $C_5$ as in Example \ref{eg_C5}. When we take two vertices, say $1$ and $3$, we have a $3\times4$ submatrix $[1,3]\times [3,1]$ of $\Sigma$ as below:
\[
\begin{pmatrix}
\sigma_{11}&\sigma_{13}&\sigma_{14}&\sigma_{15}\\
\sigma_{21}&\sigma_{23}&\sigma_{24}&\sigma_{25}\\
\sigma_{31}&\sigma_{33}&\sigma_{34}&\sigma_{35}
\end{pmatrix}
\]
and all the $3\times3$ minors provide equations of its Sullivant-Talaska ideal $I_5$. For instance, if we choose $1,2,3$ as row indices and $3,4,5$ as column indices for a minor, then we obtain a cubic defining equation
\[\sigma_{13}\sigma_{24}\sigma_{35}-\sigma_{13}\sigma_{25}\sigma_{34}-
\sigma_{23}\sigma_{14}\sigma_{35}+
\sigma_{23}\sigma_{15}\sigma_{34}+
\sigma_{33}\sigma_{14}\sigma_{25}-
\underline{\sigma_{33}\sigma_{15}\sigma_{24}},\]
where the underlined monomial is leading with respect to a term order in
Definition \ref{def:order}. 
\end{ex}

\begin{rem}
Ideals of minors of symmetric matrices have been intensively studied, also by choosing a good term order. For example, Conca \cite{Conca} described properties of the ideal generated by \textit{all} the minors of a fixed size, where the order comes from reading the variables row by row. Such order, as we checked in our experiments, leads to a square-free Gr\"obner basis also for $I_n$, however not a cubic one. Additional generators in degree five appear, which makes the analysis that is about to follow extremely hard to carry out. As we will see, with respect to the term order defined above, the degree three generators alone form a Gr\"obner basis.
\end{rem}
\begin{defi}
We identify the set $\{1,\dots, n\}$ with vertices of a regular $n$-gon. A \emph{secant} is a multiset of two vertices. In particular, it may be a loop, an edge of an $n$-gon or a classical secant known in geometry.
\end{defi}

We identify $\sigma_{i,j}$ with a secant of a regular $n$-gon, joining $i$-th and $j$-th vertex (in particular, when $i=j$ we obtain a loop, a degenerate one). Accordingly, a multiset of secants corresponds to a monomial in $R$. In the following when we depict a subset of secants of the regular $n$-gon, we will suppress the $n$-gon edges for clarity.
\begin{defi}
We call $k$ distinct secants $\alpha_1=\{b_1,c_1\},\dots,\alpha_k=\{b_k,c_k\}$ of the $n$-gon a \emph{forbidden $k$-tuple} if:
\begin{itemize}
\item their ending points are distinct, i.e.~if $i\neq j$ then $\{b_i,c_i\}\cap\{b_j,c_j\}=\emptyset$
and,
\item one may cut the $n$-gon with a line going through two vertices $a,b$ so that all $b_i$'s lie on one side of the line (including $a,b$) and all $c_i$'s lie on the other side (also including $a,b$). Further, going clockwise from $a$ we encounter $b_1,b_2,\dots,b_k$ in that order and if we go counterclockwise we encounter $c_1,c_2,\dots,c_k$ in this order. 
\end{itemize}
\end{defi}



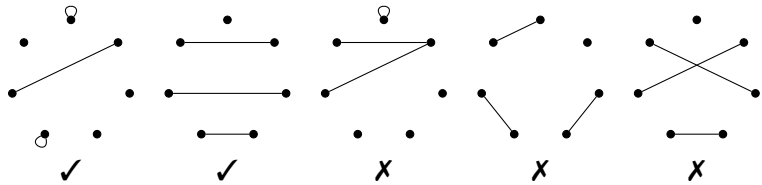
\begin{figure}[htb!]
\begin{center}
\begin{tikzpicture}[scale=0.8]
\input{triple_examples.tex}
\end{tikzpicture}
\end{center}
\caption{Examples (on the left) and non-examples (on the right) of forbidden triples for $n=7$}
\end{figure}

\begin{rem}\label{rem:middle}
  Equivalently, a triple ($k=3$) is forbidden if the secants are
  non-intersecting and one forms the \emph{middle}, that is one divides the
  $n$-gon into two regions each containing one secant in the triple. 
\end{rem}
\begin{lemma}\label{lem:leading}
The leading monomials of the $3\times 3$ minors generating $I_n$ with respect to a term order in Definition \ref{def:order} correspond to forbidden triples. 
\end{lemma}
\begin{proof}
We start by proving a general fact as follows. 

\underline{Claim:} Given any subdivision of an $n$-gon by a line through $a,b$ and a choice of distinct vertices $b_1,\dots, b_k$ on one side in the clockwise order and $c_1,\dots, c_k$ on the other side in the counterclockwise order the largest variable $x_{b_i,c_j}$ is either $x_{b_1,c_1}$ or $x_{b_k,c_k}$. 
\begin{proof}[Proof of the claim]
We note that $|p-q|$ is the number of vertices we pass on an $n$-gon while joining $p$ and $q$, not passing through the edge $\{1,n\}$. If we do not start from $b_1$ or $b_k$ and go to $c_j$ we have to pass through either $b_1$ or $b_k$ on our way, thus we cannot be minimal in length. Also the path from $b_1$ to $c_k$ (resp.~from $b_k$ to $c_1$) has to pass through other $b_i$'s or $c_j$'s, that is not a minimal one. Thus, the only remaining possibilities are the two given ones in the claim. \phantom\qedhere 
\end{proof}
For any $3\times 3$ minor $m$ in $I_n$ (say $m$ is a minor of a submatrix $[a,b]\times [b,a]$ and rows indexed by $b_1,b_2,b_3$ clockwisely and columns by $c_1,c_2,c_3$ counterclockwisely), by Laplace expansion, we note the monomials appearing in $m$ correspond to perfect pairings on the bipartite $K_{3,3}$ given by $b_i$'s and $c_i$'s. By the claim above, we see that choosing the largest variable (among the edges of $K_{3,3}$) corresponds to choosing the secant `closest' to $a$ or $b$. Since our term order is lexicographic, choosing the largest variable each time inductively leads to the leading monomial of $m$ whose secants do not intersect and form a forbidden triple. 

Conversely, given a forbidden triple $\alpha_1=\{b_1,c_1\},\alpha_2=\{b_2,c_2\},\alpha_3=\{b_3,c_3\}$, where $a,b$ justify it is forbidden, let us consider the $3\times 3$ minor with rows $b_1,b_2,b_2$ and columns $c_1,c_2,c_3$. Then, it is straightforward to see that this minor is among the generators of $I_n$ as it is a minor of the submatrix $[a,b]\times [b,a]$. 
\end{proof}
We will obtain a computer-free proof that all these $3\times 3$ minors form a Gr\"obner basis of $I(C_n)$ in the end (see Remark \ref{GBness}). Below, we present a short proof of a weaker fact, that is based on extensive computations.  
\begin{prop}
The $3\times 3$ minors form a Gr\"obner basis of the ideal $I_n$ for a term order $<$ defined in Definition \ref{def:order}. In particular, the cubic square free monomials corresponding to forbidden triples generate $\initial(I_n)$. 
\end{prop}
\begin{proof}
For $n\leq 10$ the statement can be checked using computational methods. We note that in particular, by Buchberger's algorithm, all S-pairs of minors reduce to zero.

Fix $n>10$ and consider two $3\times 3$ minors $m_1$, $m_2$. If these minors do not share variables then we do not have to consider the S-pair by Buchberger's second criterion. However, if they do share a variable, they belong to a $10\times 10$ principal submatrix of $K$, corresponding to a $10$ element subset of $\{1,\dots, n\}$ containing all row and column indices of $m_1$ and $m_2$. As our term order restricts to a term order of the same kind for principal submatrices, we conclude that the S-pair of $m_1$ and $m_2$ also reduces to zero by induction.
\end{proof}

\section{Equidimensionality}\label{sec:eqdim}
Our next aim is understanding of the variety defined by the radical monomial ideal $J$ which is generated by the square free leading terms of the $3\times 3$ minors generating $I_n$.
\begin{defi}
A \emph{maximal set avoiding forbidden triples (Msaft)} is an inclusion maximal subset of secants that does not contain any forbidden triple.
\end{defi}

For instance, in Figure \ref{AllMsafts} we present all possible Msafts in case of $n=4, 5$. Note that from the definition you will find at least one forbidden triple when you add one more secant to a given Msaft.

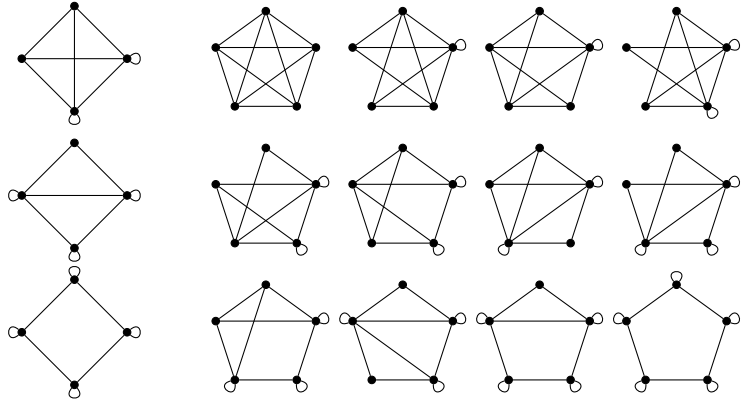
\begin{figure}[hbtp!]
\begin{center}
\begin{tikzpicture}[scale=0.7]
\input{plot4.tex}
\end{tikzpicture} \qquad
\begin{tikzpicture}[scale=0.7]
\input{plot5.tex}
\end{tikzpicture}
\end{center}
\caption{All Msafts  up to dihedral symmetries for $n=4$ (left) and $n= 5$ (right)}
\label{AllMsafts}
\end{figure}

\begin{prop}\label{prop:bijectionMSaft}
The variety $V(J)$ is a union of coordinate subspaces, each one spanned by basis vectors corresponding to elements of an Msaft. In particular, there is a bijection between irreducible components of $V(J)$ and Msafts.
\end{prop}
\begin{proof}
The zero set of a monomial ideal is always a union of coordinate subspaces, spanned by inclusion maximal sets of basis vectors, not containing the dual vectors to support of monomial generators.
The proposition is thus a corollary of Lemma \ref{lem:leading}.
\end{proof}
Our next aim is to better understand the combinatorial properties of Msafts and build a bijection with lattice walks circling twice a M\"obius strip. We start with a useful `moving lemma' as follows.
\begin{lemma}\label{lem:oneoftwo}
Suppose a secant $\alpha:=\{i,j\}$ belongs to an Msaft $M$ (where possibly $i=j$). Then at least one of $\beta:=\{i,j+1\}$ and $\gamma:=\{i+1,j\}$ and at least one of
$\beta':=\{i-1,j\}$ and $\gamma':=\{i,j-1\}$ must belong to $M$. Here, addition and subtraction is modulo $n$.

If additionally, $\alpha':=\{i-1,j+1\}$ belongs to $M$, then at least
one of $\delta := \{i-1,j+2\}$ and $\gamma$ belongs to $M$, and at least one of
$\delta' := \{i-2,j+1\}$ and $\gamma'$ belongs to $M$.
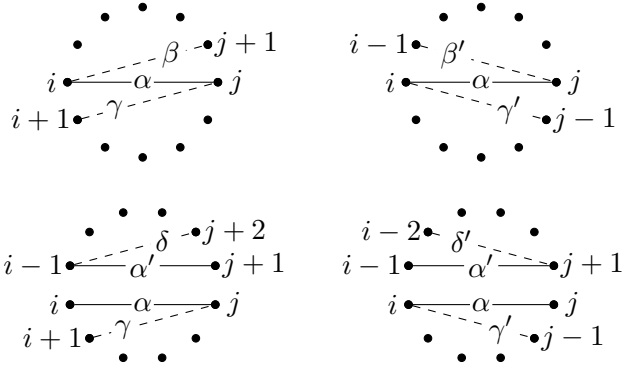
\begin{center}
\begin{tikzpicture}
\tikzset{every node/.style={circle,fill,draw=black,minimum size = 0.1cm,inner sep=0cm}}
\begin{scope}
\foreach \i in {0,...,11}
\node (\i) at (90+30*\i:1cm) {};
\node [label=left:$i$] at (3) {};
\node [label=left:$i+1$] at (4) {};
\node [label=right:$j$] at (9) {};
\node [label=right:$j+1$] at (10) {};
\draw (3) -- (9) node [midway,draw=none,fill=white] {$\alpha$};
\draw [dashed] (4) -- (9) node [near start,draw=none,fill=white] {$\gamma$};
\draw [dashed] (3) -- (10) node [near end,draw=none,fill=white] {$\beta$};
\end{scope}
\begin{scope}[shift={(4.5cm,0cm)}]
\foreach \i in {0,...,11}
\node (\i) at (90+30*\i:1cm) {};
\node [label=left:$i$] at (3) {};
\node [label=left:$i-1$] at (2) {};
\node [label=right:$j$] at (9) {};
\node [label=right:$j-1$] at (8) {};
\draw (3) -- (9) node [midway,draw=none,fill=white] {$\alpha$};
\draw [dashed] (2) -- (9) node [near start,draw=none,fill=white] {$\beta'$};
\draw [dashed] (3) -- (8) node [near end,draw=none,fill=white] {$\gamma'$};
\end{scope}
\begin{scope}[shift={(0cm,-2.7cm)}]
\foreach \i in {0,...,11}
\node (\i) at (75+30*\i:1cm) {};
\node [label=left:$i-1$] at (3) {};
\node [label=left:$i$] at (4) {};
\node [label=left:$i+1$] at (5) {};
\node [label=right:$j$] at (9) {};
\node [label=right:$j+1$] at (10) {};
\node [label=right:$j+2$] at (11) {};
\draw (4) -- (9) node [midway,draw=none,fill=white] {$\alpha$};
\draw (3) -- (10) node [midway,draw=none,fill=white] {$\alpha'$};
\draw [dashed] (5) -- (9) node [near start,draw=none,fill=white] {$\gamma$};
\draw [dashed] (3) -- (11) node [near end,draw=none,fill=white] {$\delta$};
\end{scope}
\begin{scope}[shift={(4.5cm,-2.7cm)}]
\foreach \i in {0,...,11}
\node (\i) at (75+30*\i:1cm) {};
\node [label=left:$i-2$] at (2) {};
\node [label=left:$i-1$] at (3) {};
\node [label=left:$i$] at (4) {};
\node [label=right:$j-1$] at (8) {};
\node [label=right:$j$] at (9) {};
\node [label=right:$j+1$] at (10) {};
\draw (4) -- (9) node [midway,draw=none,fill=white] {$\alpha$};
\draw (3) -- (10) node [midway,draw=none,fill=white] {$\alpha'$};
\draw [dashed] (4) -- (8) node [near end,draw=none,fill=white] {$\gamma'$};
\draw [dashed] (2) -- (10) node [near start,draw=none,fill=white] {$\delta'$};
\end{scope}
\end{tikzpicture}
\end{center}
\end{lemma} 

We note that in case $i=j$, i.e.~$\alpha$ is a loop, the lemma implies that $M$ contains the two adjacent edges, as then $\beta=\gamma$ and $\beta'=\gamma'$. 
\begin{proof}
  We present the proof in the case $\alpha$ and $\alpha'$ are neither edges nor loops. In these remaining cases the argument is left to the reader.
  
  Let us name (arbitrarily) the two sides of the $n$-gon bisected by $\alpha$ the
  \emph{top}, containing $j+1$, and the \emph{bottom}, containing $i+1$. 
  Suppose neither
  $\beta$ nor $\gamma$ belong to $M$. Then there must be a forbidden triple
  containing $\beta$ and two secants belonging to $M$, and similarly for
  $\gamma$. Let $b \in \{0,1,2\}$ be the number of secants belonging to $M$
  lying in the top side of the $n$-gon in the triple containing $\beta$, and
  let $c\in \{0,1,2\}$ be the number of secants belonging to $M$ lying in the
  bottom side of the $n$-gon in the triple containing $\gamma$. If $b+c \ge 2$,
  then any size three subset of the corresponding secants and $\alpha$
  is a forbidden triple of secants belonging to $M$. If $b+c <2$, then
  $(2-b)+(2-c) \ge 3$ counts the set of secants in the triple with and below
  $\beta$ and those in a triple with and above $\gamma$, any size three subset
  of which is a forbidden triple of secants belonging to $M$. In either case,
  we have a forbidden triple consisting of secants belonging to $M$, which is a
  contradiction, thus, $\beta$ or $\gamma$ must belong to $M$. Similarly,
  $\beta'$ or $\gamma'$ must belong to $M$.

  The proof of the second claim is similar. Suppose neither $\delta$ nor
  $\gamma$ belong to $M$, so that each is contained in a forbidden triple with
  two secants belonging to $M$. Let $b$ be the number of secants in the triple
  for $\delta$ in the top side of the $n$-gon, and let $c$ be the number of
  secants in the triple for $\gamma$ in the bottom side. If $b+c \ge 1$, such a
  secant with $\alpha$ and $\alpha'$ is a forbidden triple of secants belonging
  to $M$. If $b=c=0$, then the bottommost secant in the triple corresponding to
  $\delta$, the topmost secant in the triple corresponding to $\gamma$, and
  $\beta$ form a forbidden triple. Since by the previous paragraph $\beta$
  belongs to $M$, we obtain a contradiction in either case, and the claim is
  proved. By a similar argument one of $\delta'$ and $\gamma'$ must belong to
  $M$.
\end{proof}

Partition the set of secants $\{i,j\}$ into $n$ \emph{parallel families} according to the value
of $i+j$ modulo $n$. Geometrically this corresponds to collecting together
parallel secants and assigning loops appropriately (see Figure
\ref{fig:parallel families}). A set of three secants with the same value of
$i+j$ forms a forbidden triple, so an Msaft $M$ contains at most two secants in
the same parallel family. We will show there are exactly two.

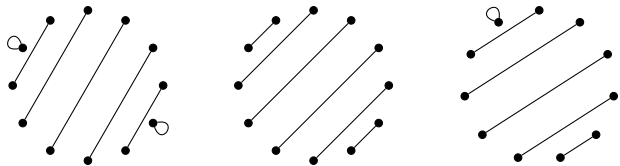
\begin{figure}[hbtp!]
\begin{center}
\begin{tikzpicture}
\tikzset{every node/.style={circle,fill,draw=black,minimum size = 0.1cm,inner sep=0cm}}
\begin{scope}
\foreach \i in {0,...,11}
\node (\i) at (90+30*\i:1cm) {};
\path[-,every loop/.style={min distance=0.3cm}]
 (2) edge [in=90+2*30-40,out=90+2*30+40,loop] ()
 (3) edge (1)
 (4) edge (0)
 (5) edge (11)
 (6) edge (10)
 (7) edge (9)
 (8) edge [in=90+8*30-40,out=90+8*30+40,loop] ();
\end{scope}
\begin{scope}[shift={(3cm,0cm)}]
\foreach \i in {0,...,11}
\node (\i) at (90+30*\i:1cm) {};
\path[-,every loop/.style={min distance=0.3cm}]
  (2) edge (1)
  (3) edge (0)
  (4) edge (11)
  (5) edge (10)
  (6) edge (9)
  (7) edge (8);
\end{scope}
\begin{scope}[shift={(6cm,0cm)}]
\foreach \i in {0,...,10}
\node (\i) at (90+360/11*\i:1cm) {};
\path[-,every loop/.style={min distance=0.3cm}]
 (1) edge [in=90+1*360/11-40,out=90+1*360/11+40,loop] ()
 (2) edge (0)
  (3) edge (10)
  (4) edge (9)
  (5) edge (8)
  (6) edge (7);
\end{scope}
\end{tikzpicture}
\end{center}
\caption{Examples of parallel families of secants $\{i,j\}$ with equal value of $i+j$
modulo $n$. When $n$ is even, such a family may have two loops (size $\frac{n+2}{2}$) or
zero loops (size $\frac{n}{2}$). When $n$ is odd, such a family has one loop and size
$\frac{n+1}{2}$.}
\label{fig:parallel families}
\end{figure}

Lemma \ref{lem:oneoftwo} may be interpreted as giving us relations between the
secants belonging to $M$ of two parallel families where the value of $i+j$ differs by 1. 
Say the secants $\beta = \{i,j+1\}$ and $\gamma = \{i+1,j\}$ (resp.~$\beta' =
\{i-1,j\}$ and $\gamma' = \{i,j-1\}$) are each \emph{to the right}
(resp.~\emph{to the left}) of $\alpha = \{i,j\}$. The
lemma states $M$ contains a secant to the right (resp.~left) of any
secant belonging to it. If two secants belong to $M$ and are in the same
parallel family, by the second claim of the lemma we may additionally assign the
secants to the right (resp.~left) so that they are distinct. Geometrically, a secant to
the right (resp.~left) is one turned the minimal amount counterclockwise (resp.~clockwise)
and which shares a vertex. The reason for the terminology will become more
clear in the next section.

Any Msaft $M$ contains two non-intersecting secants $\{i_1,j_1\}$ and
$\{i_2,j_2\}$ (for instance, any secant not in $M$ is contained in a forbidden
triple with two such secants). By possibly interchanging the secants and
cyclically relabelling the vertices, we may suppose $i_1\le j_1 < i_2\le j_2$
and $i_1 + j_1 +n \le i_2+j_2$. Apply Lemma \ref{lem:oneoftwo} to find a secant
belonging to $M$ $\{i_3,j_3\}$ which is $(i_2+j_2) - (i_1+j_1+n)$ times to the
right of $\{i_1,j_1\}$, and thus in the same parallel family as $\{i_2,j_2\}$. We cannot
have $ \{i_2,j_2\} = \{i_3,j_3\}$, as this would require $n$ more moves to the
right. Hence, for at least one parallel family, $M$ contains two of its secants. Applying
the lemma $n-1$ times to this parallel pair of secants as in the previous paragraph, we have proved
the following.

\begin{prop}\label{prop:Msaft2n}
  Fix $n\geq 3$. For each $k \in \{1,\ldots,n\}$, any Msaft $M$ contains
  exactly two secants of the form $\{i,j\}$ where $i+j = k$ modulo $n$. In
  particular, $M$ has cardinality $2n$. 
  Moreover, each secant belonging to $M$ for $k$ is to the left of a
  distinct secant belonging to $M$ for $k+1$.
\end{prop}

We know that $J$ is radical, as it is generated by square-free monomials, and
that the zero set is a union of coordinate linear subspaces corresponding to Msafts.
Thus, the equidimensionality of the variety $V(J)$ follows from
Proposition \ref{prop:Msaft2n}, and we have the following:

\begin{cor}\label{cor:equi}
The variety $V(J)\subset S^2\CC^n$ is $2n$ equidimensional. Its degree is equal
to the number of Msafts. 
\end{cor} 

\section{Counting Msafts and Proof of the theorem}\label{sec:countingMS}

\subsection{Counting Msafts}\label{subsec:countMS}

Since two different secants cannot each be to the right of two different
secants, in the conclusion of Proposition \ref{prop:Msaft2n} the association of
secants for $k$ to distinct secants to the right is unique. We can thus
associate to $M$ the function $f : M\to M$ taking a secant to the unique secant
to the right so that $f$ is injective.

More generally, let $M$ be any subset of secants containing exactly two from
each parallel family. A \emph{secant walk} is an injective function $f: M\to M$ where
$f(\alpha)$ is to the right of $\alpha$ for each $\alpha$. Of course, not every
such subset $M$ admits a walk; in fact, we have the following.

\begin{prop}\label{prop:walk}
  The underlying sets of secant walks are precisely the Msafts.
\end{prop}

To prove Proposition \ref{prop:walk} and count Msafts, we introduce a more
convenient visual representation for sets admitting secant walks which allows
easier reasoning. Specifically, we represent each of the $\binom{n+1}{2}$
secants as a point, and we arrange them so that those in the same parallel family
are in the same column and so that a secant to the right of $\alpha$ is to the
top right or bottom right of $\alpha$ (which finally justifies our choice of
naming of this relation). Such an arrangement is only possible on a M\"obius
strip running left to right, so on the page we imagine the column one past the
end of the diagram is the first column again, but reversed top to bottom. See
Figure \ref{fig:msaft walks} for such diagrams corresponding to all Msafts for $n=4,5$.

\begin{figure}[hbtp!]
\begin{center}
\begin{tikzpicture}[scale=0.8]
\input{plot4edge.tex}
\end{tikzpicture}

\begin{tikzpicture}[scale=0.8]
\input{plot5edge.tex}
\end{tikzpicture}
\end{center}
\caption{Secant walks corresponding to Msafts for $n=4$ (top) and $n= 5$
  (bottom). The upper left lattice point corresponds to the loop at the
  northernmost vertex, thus its column corresponds to horizontal secants.
  Moving to the right columnwise corresponds to counterclockwise rotation of
secants.}
\label{fig:msaft walks}
\end{figure}
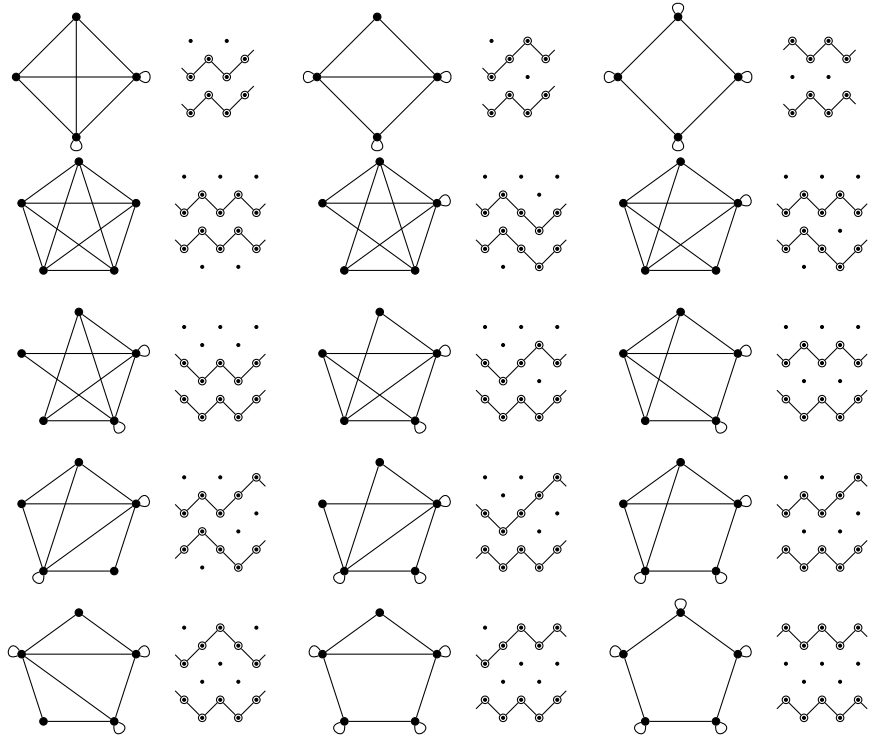

To establish our bijection of Msafts with secant walks, we need to understand
what forbidden triples look like in the diagrammatic representation. By Remark
\ref{rem:middle}, it suffices to understand for a given secant how the
non-intersecting secants are divided on each side of the $n$-gon. 
\begin{defi}
Fix a secant $\alpha$ and identify it with a lattice point on a diagram, as presented in Figure \ref{fig:regions}. There is a distinguished, connected region on the M\"obius strip containing exactly those lattice points which correspond to secants intersecting $\alpha$ (possibly sharing just a vertex). We denote it by $R^\alpha_3$ and it is the white region in Figure \ref{fig:regions}. 



Additionally, there are at most two regions, each one containing points corresponding to secants that do not intersect $\alpha$ but lie on a fixed side of $\alpha$. We denote them by $R_1^\alpha$ and $R_2^\alpha$. If $\alpha$ is an edge or a loop one of these two regions is empty.  
\end{defi}

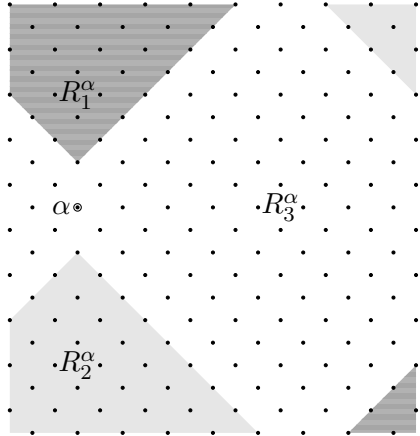
\begin{figure}[hbtp!]
\begin{center}
\begin{tikzpicture}[scale=0.3]
\input{regions.tex}
\end{tikzpicture}
\end{center}
\caption{A secant $\alpha$ intersects all the secants in region $R_3^\alpha$.
  The remaining secants are divided into regions $R_1^\alpha$ and $R_2^\alpha$
  according to which side of $\alpha$ they are on in the $n$-gon. The regions
  wrap around the sides according to the M\"obius geometry of the strip.}
\label{fig:regions}
\end{figure}

\begin{proof}[Proof of Proposition \ref{prop:walk}]
We must show the underlying set of a secant walk is an Msaft. For contradiction
assume such a set contains a forbidden triple $\{\alpha,\beta,\gamma\}$, where
$\gamma$ is the middle element, say $\alpha \in R_1^\gamma$ and $\beta\in
R_2^\gamma$. Suppose $\gamma$ lies in column $k$. Applying $f$ or $f^{-1}$ to
$\alpha$ and $\beta$, we may move each to column $k$ while staying $R_1^\gamma$
and $R_2^\gamma$, respectively. In particular, we obtain three distinct
elements of $M$ in column $k$, contradicting our assumption. The claim is
proved.
\end{proof}

\begin{rem}
  A secant walk clearly cannot ``cross" itself. It follows from the M\"obius
  geometry that in a secant walk, each vertex is reachable by each other vertex
  by iterated applications of $f$.
\end{rem}

To count the number of secant walks we will need two results from graph theory,
the Lindstr{\"o}m-Gessel-Viennot lemma and the reflection principle.

Let $G$ be a finite directed acyclic graph and let $A=\{a_1,\dots,a_n\}$ and
$B=\{b_1,\dots,b_n\}$ be disjoint subsets of vertices of $G$. An \emph{$n$-tuple of
non-intersecting paths from $A$ to $B$} is a tuple $P=(P_1,\dots,P_n)$ of paths
in $G$ satisfying 
\begin{enumerate}
\item There is a permutation $\sigma(P) :=\sigma \in S_n$ such that $P_i$ is a
  path from $a_i$ to $b_{\sigma(i)}$.
\item $P_i$ and $P_j$ have no common vertex whenever $i\neq j$. 
\end{enumerate}
Consider the following $n\times n$ matrix
\begin{equation}\label{nonintersectMatrix}
M(G:A,B):=\begin{pmatrix}
e(a_1,b_1)&\cdots&e(a_1,b_n)\\
\vdots&\ddots&\vdots\\
e(a_n,b_1)&\cdots&e(a_n,b_n)
\end{pmatrix},
\end{equation}
where $e(a,b):= \# \{ \text{paths $a\to b$ in $G$} \}$. In this situation,
the Lindstr{\"o}m-Gessel-Viennot lemma (\cite{GV, Lind}) in unweighted form
states $\det M(G:A,B)= \sum_{P} \operatorname{sign}(\sigma(P))$, where the sum is
taken over $n$-tuples of non-intersecting paths from $A$ to $B$. 

We wish to express the problem of counting secant walks as one of counting
non-intersecting paths on a graph. Distinguish a column $C$ in the M\"obius
strip, and form the graph $G$ with vertices the set of secants and edges
$\alpha\to \beta$ when $\beta$ is to the right of $\alpha$, except that a
secant in $C$ occurs as vertex twice, once with only its in-edges and once with
only its out-edges. Write $C\subset G$ as the copy with only out-edges and $C'$
for the copy with only in-edges. In effect, we take our usual M\"obius diagram,
make $C$ the first column, duplicate it as one past the $n$th column to obtain
$C'$ (flipping top and bottom), and take edges from each vertex to those north
east and south east from it when they exist. 

For a fixed subset $A$ of size $2$ of $C$, put $A'$ as this subset inside $C'$.
It is clear that secant walks intersecting $C$ in $A$ are in bijection with
non-intersecting pairs of paths from $A$ to $A'$ in $G$. Since paths from $A$
to $A'$ cannot cross, $\sigma = (1,2)$ is the only possible permutation which
can occur in the sum, so that by the Lindstr{\"o}m-Gessel-Viennot lemma, secant
walks intersecting $C$ in $A$ are counted by $-\det M(G,A,A')$, hence all secant
walks are counted by $\sum_A -\det M(G,A,A')$, with the sum taken over all size 2
subsets of $C$.
Thus, to count secant walks, it remains to compute the entries of $M(G,A,A')$,
that is, the number of paths from each element of $C$ to each element of $C'$.
To make our notation concrete, we now fix $C$ as a column containing a loop.
Such a column has $\lfloor \frac{n}{2}\rfloor + 1$ elements, which we label
with $\{0,\ldots,\lfloor \frac{n}{2}\rfloor\}$ in order starting with the loop,
and $C'$ correspondingly, see Figure \ref{fig:graph}.

\begin{figure}
\begin{center}
\begin{tikzpicture}[scale=1]
\input{plot6one.tex}
\end{tikzpicture}
\qquad
\begin{tikzpicture}[scale=1]
\input{plot7one.tex}
\end{tikzpicture}
\end{center}
\caption{Examples of disjoint path pairs and the indexing of the first column
and its copy the last}
\label{fig:graph}
\end{figure}
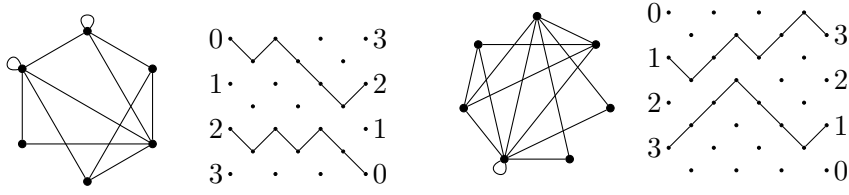

Write $e(i,j)$ for the number of paths from $i\in C$ to $j\in C'$ in $G$. To
count these, we first model $G$ as a finite subset of a lattice, namely
\[
  G = \Gamma \cap \{ (x,y) \mid -n \le y \le 0 \},
\]
where
\[
  \Gamma = \{\ZZ (1,1) + \ZZ (1,-1)\} \cap \{ (x,y) \mid 0 \le x \le n \}.
\]
We equip the infinite vertical strip $\Gamma$ with the structure of a directed
graph by saying there is an edge from $u$ to $v$ if $v=u + (1,1)$ or $v = u+
(1,-1)$. The existing graph structure for $G$ is given the same way. We will
compute $e(i,j)$ by first computing $e'(i,j)$, the number of paths from $i$ to
$j$ \emph{through $\Gamma$}. Here $i, j \in \ZZ$, with $i$ corresponding to
lattice point $(0,-2i)$ and $j$ corresponding to lattice point $(n,-n+2j)$.

In what follows we set $\binom{n}{k}=0$ when $k<0$ or $n<k$. 
\begin{lemma}\label{lem:unbounded}
For $i,j \in \ZZ$, $e'(i,j)=\binom{n}{i+j}$.
\end{lemma}
\begin{proof}
Starting at $(0,-2i)$, if we make $k$ steps in direction $(1,1)$ and hence
$n-k$ steps in direction $(1,-1)$ we end in position $(n,-2i+2k-n) =
(n,-n+2j)$, hence $k=i+j$. Paths through $\Gamma$ from $(0,-2i)$ to $(n,-n+2j)$
are in bijection with the number of ways to choose the $k$ positions of the
$(1,1)$ steps from $n$, and the lemma follows.
\end{proof}
\begin{lemma}\label{lem:how_many_intersecting}
  For $0\le i,j \le \lfloor \frac{n}{2}\rfloor$, $e(i,j) = \binom{n}{i+j}-\binom{n}{j-i-1}-\binom{n}{i-j-1}$. 
\end{lemma} 
\begin{proof}
Let $P^\Gamma(i,j)$ and $P^G(i,j)$ be the set of paths from $i$ to $j$
through $\Gamma$ and $G$, respectively, and let $P^{\text{top}}(i,j)$ and
$P^{\text{bottom}}(i,j)$ be the set of paths through $\Gamma$ from $i$ to
$j$ which touch the lines $y=1$ and $y=-n-1$, respectively.

First we note that a path through $\Gamma$ has $n$ steps, so it cannot touch
both the lines $y=1$ and $y=-n-1$. Hence we have the disjoint union
\[
  P^\Gamma(i,j) = P^G(i,j) \cup P^{\text{top}}(i,j) \cup P^{\text{bottom}}(i,j).
\]
We form a bijection between $P^{\text{bottom}}(i,j)$ and $P^\Gamma(i, -j-1)$
and between $P^{\text{top}}(i,j)$ and $P^\Gamma(-i-1,j)$ using the
\emph{reflection principle}.

Consider a lattice path $p \in P^{\text{bottom}}(i,j)$. Let $k$ be the first
step $p$ touches the line $y=-n-1$. Let $p'$ be $p$, but with all steps
reversed starting from step $k+1$ to the end. Geometrically, this has the effect of
reflecting the path corresponding to $p$ accross the line $y=-n-1$ in this
region. In particular $p'$ ends at $-j-1$. Applying the same rule to $p'$ gives
us $p$ again, and we see this assignment is bijective.
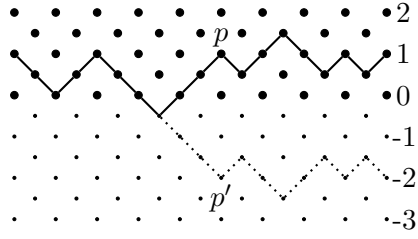
\begin{center}
\begin{tikzpicture}[scale=0.55]
\input{lattice.tex}
\end{tikzpicture}
\end{center}
Similarly, we reverse the steps of a lattice path $p\in P^{\text{top}}(i,j)$
from the first step up to and including the first step $p$ touches the line
$y=1$ to obtain a path $p'\in P^{\Gamma}(-i-1,j)$, and this map is a bijection.
Applying Lemma \ref{lem:unbounded} completes the proof.
\end{proof}








\begin{prop}\label{prop:how_many_Msaft}
For any $n\geq 4$ there are
\[\frac{1}{2}\sum_{i=0}^{\lfloor \frac{n}{2}\rfloor}\sum_{j=0}^{\lfloor \frac{n}{2}\rfloor}\left(\binom{n}{i+j}-\binom{n}{i-j-1}-\binom{n}{j-i-1}\right)^2-\binom{n}{2i}\binom{n}{2j}\]
Msafts.
\end{prop}
\begin{proof}
  By the discussion before Lemma \ref{lem:unbounded}, from the
  Lindstr{\"o}m-Gessel-Viennot Lemma, we have that the number of secant walks
  are counted by $\sum_A -\det M(G,A,A')$, with the sum being taken over all
  size 2 subsets of $C$. We have indexed elements of $C$ by $\{0,\ldots,\lfloor
  \frac{n}{2}\rfloor\}$, so this sum becomes
  \[
    \sum_{0\le i < j \le \lfloor \frac{n}{2} \rfloor} -\det \begin{bmatrix} e(i,i) & e(i,j) \\ e(j,i) & e(j,j) \end{bmatrix}.
  \]
  The summand vanishes when $i = j$ and is invariant under interchanging $i$
  and $j$, so we may replace the sum with half the sum where $i$ and
  $j$ each range over $\{0,\ldots,\lfloor \frac{n}{2} \rfloor\}$. Applying
  Lemma \ref{lem:how_many_intersecting} we obtain the sum in the claim, which
  is the number of Msafts by Proposition \ref{prop:walk}.
\end{proof}

\subsection{Degree and some binomial identities}\label{sec:binoms}

\begin{prop}\label{prop:binomialeq}
For any positive integer $n$ we have:
\begin{multline*}
\sum_{i=0}^{\lfloor \frac{n}{2}\rfloor}\sum_{j=0}^{\lfloor \frac{n}{2}\rfloor}\left(\binom{n}{i+j}-\binom{n}{i-j-1}-\binom{n}{j-i-1}\right)^2-\binom{n}{2i}\binom{n}{2j}\\
{}= \frac{n+2}{2}\binom{2n}{n}-3\cdot 2^{2n-2}.
\end{multline*}
\end{prop}
Before we prove the proposition we will need a few simple lemmas. We start with the following that is well-known.
\begin{lemma}\label{lem:easy}
\begin{enumerate}
  \item\label{item:1} $\sum_{i=0}^n \binom{n}{i}=2^n$,
\item\label{item:2} $\sum_{i=0}^n \binom{n}{i}^2=\binom{2n}{n}$,
\item\label{item:3} $\sum_{i=0, i\text{ even}}^n \binom{n}{i}=\sum_{i=1, i\text{ odd}}^n \binom{n}{i}=2^{n-1}$.
\end{enumerate}
\end{lemma}
\begin{proof}
\begin{enumerate}
\item Both sides compute the number of subsets of a set with $n$ elements.
\item We have: $\sum_{i=0}^n \binom{n}{i}^2=\sum_{i=0}^n \binom{n}{i}\binom{n}{n-i}$. The equality $\sum_{i=0}^n \binom{n}{i}\binom{n}{n-i}=\binom{2n}{n}$ follows by comparing the coefficient of $x^ny^n$ in the expression $(x+y)^n(x+y)^n=(x+y)^{2n}$. 
\item Fix an element $a$ of a set $S$ with $n$ elements. Divide all subsets of $S$ into pairs: $(P,Q)$ where $a\in P$ and $Q=P\setminus\{a\}$. Every pair contains precisely one odd and one even subset, thus there are as many even as odd subsets of $S$. The first equality follows. The second equality is now a consequence of \eqref{item:1}. 
\end{enumerate}
\end{proof}
\begin{lemma}\label{lem:binomtail}
\[\sum_{i=0}^{\lfloor \frac{n}{2}\rfloor}\sum_{j=0}^{\lfloor \frac{n}{2}\rfloor} \binom{n}{2i}\binom{n}{2j}=2^{2n-2}.\]
\end{lemma}
\begin{proof}
We have: 
\[\sum_{i=0}^{\lfloor \frac{n}{2}\rfloor}\sum_{j=0}^{\lfloor \frac{n}{2}\rfloor} \binom{n}{2i}\binom{n}{2j}=\left(\sum_{i=0}^{\lfloor \frac{n}{2}\rfloor} \binom{n}{2i}\right)^2.\]
The statement follows from Lemma \ref{lem:easy}\eqref{item:3}. 
\end{proof}
Let us define:
\[h(i,j)=\begin{cases}
i-j-1\text{ for } i\geq j\\
j-i-1\text{ for } i<j.
\end{cases}\]
\begin{lemma}\label{lem:kwadraty}
\begin{align*}
\sum_{i=0}^{\lfloor \frac{n}{2}\rfloor}\sum_{j=0}^{\lfloor \frac{n}{2}\rfloor}\left(\binom{n}{i+j}^2+\binom{n}{h(i,j)}^2\right)=
\begin{cases}
\frac{n+2}{2}\binom{2n}{n}\text{ for }n\text{ even}\\
\frac{n}{2}\binom{2n}{n}\text{ for }n\text{ odd}.
\end{cases}
\end{align*}
\end{lemma}
\begin{proof}
Let $a< \lfloor \frac{n}{2}\rfloor$. We note that the binomial $\binom{n}{a}^2$ appears in the sum above $(a+1)$ times as $\binom{n}{i+j}^2$ and $2(\lfloor \frac{n}{2}\rfloor-a)$ as $\binom{n}{h(i,j)}^2$. The binomial $\binom{n}{n-a}^2$ appears exactly $(2\lfloor \frac{n}{2}\rfloor-n+a+1) $ many times.

As both binomials are equal we count how many times one of them appears: 
\[\textstyle a+1+ 2(\lfloor \frac{n}{2}\rfloor-a)+(2\lfloor \frac{n}{2}\rfloor-n+a+1)=4\lfloor \frac{n}{2}\rfloor+2-n.\]

When $n$ is even this equals $n+2$. In this case $\binom{n}{\frac{n}{2}}$ appears $\frac{n+2}{2}$ many times. Thus the total sum equals:
\[\frac{n+2}{2}\sum_{i=0}^n \binom{n}{i}^2=\frac{n+2}{2}\binom{2n}{n},\]
where the last equality follows from Lemma \ref{lem:easy}\eqref{item:2}. 

When $n=2k+1$ then together the two binomials $\binom{n}{a}^2$ and $\binom{n}{n-a}^2$ appear
\[4k+2-2k-1=2k+1=n\]
many times and this remains true for $a=k$. Thus the total sum is equal to:
\[\frac{n}{2}\sum_{i=0}^n \binom{n}{i}^2=\frac{n}{2}\binom{2n}{n}.\]
\end{proof}
\begin{lemma}\label{lem:mieszane}
\begin{align*}
\sum_{i=0}^{\lfloor \frac{n}{2}\rfloor}\sum_{j=0}^{\lfloor \frac{n}{2}\rfloor}\binom{n}{i+j}\binom{n}{h(i,j)}=
\begin{cases}
2^{2n-2}\text{ for }n\text{ even}\\
2^{2n-2}-\frac{1}{2}\binom{2n}{n}\text{ for }n\text{ odd}.
\end{cases}
\end{align*}
\end{lemma}
\begin{proof}
First assume $n$ is even. For fixed $i$, the $i+j$ varies through the set $S_i=\{i,\dots,i+\frac{n}{2}\}$. At the same time (using the fact that $\binom{n}{a}=\binom{n}{n-a}$) $h(i,j)$ varies through the complement of $S_i$ in $\{-1,0,\dots,n\}$. Thus, instead of summing over $i,j$ we may sum over partitions $(S_i, S_i')$ of the set $\{-1.\dots,n\}$ into two, cyclically consecutive subsets with $\frac{n}{2}+1$ elements each. We put the elements of $\{-1,\dots,n\}$ an a circle turning $(S_i, S_i')$ into subdivision of the circle into two halves. For each such subdivision we take the sum of products of: $\binom{n}{a}$ where $a$ is the $l$-th coefficient counted clockwise in one part and $\binom{n}{b}$ where $b$ is the $l$-th coefficient counted counterclockwise in the second part. 

This means that each $\binom{n}{i}$ is multiplied with $\binom{n}{i-1}$ (when $i$ is first in its part), with $\binom{n}{i-3}$ (when $i$ is second in its part) etc. Summing up, it is multiplied by the sum of either odd or even binomial coefficients (depending if $i$ is even or odd). Thus we obtain, using Lemma \ref{lem:easy}\eqref{item:3}, that twice the given sum equals:
\[\sum_{i=0}^n \binom{n}{i}2^{n-1}=2^{2n-1},\]
which finishes the proof in the case when $n$ is even.

Suppose now $n=2k+1$. In this case the sets $(S_i, S_i')$ do not form a partition of $\{-1,\dots,n\}$, as $i+k+1$ does not belong to any of the sets. Also we do not get all pairs of (consecutive, $k+1$-elment) sets, as the missing number is always from $k+1$ to $n$. Let us call the sum we are computing $x$. First we double $x$, by changing each $\binom{n}{a}\binom{n}{b}$ to $\binom{n}{n-a}\binom{n}{n-b}$. This reflects the sets $(S_i,S_i')$, so that now indeed we obtain each pair of consecutive, $k+1$-element sets, apart from $(\{k+1,\dots,n\},\{0,\dots,k\})$. The sum is now $2x$. Adding the sum corresponding to the missing pair, we obtain:
\[2x+\sum_{i=0}^k \binom{n}{k+1+i}\binom{n}{k-i}=2x+\frac{1}{2}\binom{2n}{n}.\]
As before we put the elements of $\{-1,\dots,n\}$ on a circle identifying $i$ with $\binom{n}{i}$. We draw an edge between $i$ and $j$ each time we have $\binom{n}{i} \binom{n}{j}$ in the sum we consider. We note that we obtain exactly all edges, but no loops. We double our sum obtaining $4x+\binom{2n}{n}$, and now each edge is counted twice. We also add loops, i.e.~the sum $\sum_{i=-1}^n \binom{n}{i}^2$, obtaining $4x+2\binom{2n}{n}$. As we obtained all edges twice and loops once the sum we currently have is equal to:
\[\left(\sum_{i=0}^n\binom{n}{i}\right)^2=2^{2n}.\]
Solving for $x$ we obtained the claimed equality. 
\end{proof}

We now prove the main proposition about the sums of binomial coefficients.
\begin{proof}[Proof of Proposition \ref{prop:binomialeq}]
We note the LHS is equal to:
\begin{multline*}
  \left(\sum_{i=0}^{\lfloor \frac{n}{2}\rfloor}\sum_{j=0}^{\lfloor \frac{n}{2}\rfloor}\binom{n}{i+j}^2+\binom{n}{h(i,j)}^2\right)-2\left(\sum_{i=0}^{\lfloor \frac{n}{2}\rfloor}\sum_{j=0}^{\lfloor \frac{n}{2}\rfloor} \binom{n}{i+j}\binom{n}{h(i,j)}\right) \\
  {}-\left(\sum_{i=0}^{\lfloor \frac{n}{2}\rfloor}\sum_{j=0}^{\lfloor \frac{n}{2}\rfloor}\binom{n}{2i}\binom{n}{2j}\right).
\end{multline*}
Applying Lemma \ref{lem:kwadraty} to the first term, Lemma \ref{lem:mieszane} to the second and Lemma \ref{lem:binomtail} to the last one we obtain the claimed statement.  
\end{proof}

\subsection{Proof of the main theorem}\label{sec:pf_thm}

We finish the article by proving our main theorem.
\begin{proof}[Proof of Theorem \ref{thm:main}]
As mentioned in \S\ref{sec:cyclic}, the statement holds trivially for $n=3$. For any $n\ge4$, we will apply Lemma \ref{lem:general}, taking $g_i$'s to be the $3\times 3$ minors generating the Sullivant-Talaska ideal $I_n$ and the term order specified in Definition \ref{def:order}. We need to verify the three conditions. The first one follows from Lemma \ref{lem:leading}. The second one and the dimension part of the third one is precisely Corollary \ref{cor:equi}. For the degree part by the same corollary we know that the degree of $J$ equals the number of Msafts. Combining Proposition \ref{prop:how_many_Msaft} with Proposition \ref{prop:binomialeq} we see that it equals:
\[\frac{n+2}{4}\binom{2n}{n}-3\cdot 2^{2n-3}.\]
By the main theorem of \cite{MRV} this is also the degree of the cyclic Gaussian graphical model, which finishes the proof. 
\end{proof}

\begin{rem}\label{GBness}
As we proved the initial ideal of $I_n$ coincides with the initial ideal of $I(C_n)$, the prime ideal of the Gaussian graphical model of the $n$-cycle. Thus, the special $3\times 3$ minors form a Gr\"obner basis of $I(C_n)$ for a term order defined in Definition \ref{def:order}. 
\end{rem}
As a final remark, we note that in our proof we use a complicated result computing the degree of $L^{-1}_{C_n}$, which was only achieved very recently. However, our proof actually shows that if one could prove our main Theorem \ref{thm:main} without referring to this result, it would also provide an alternative, possibly simpler, proof for the degree of $L^{-1}_{C_n}$.

\end{document}

%% file: triple_examples.tex
\begin{scope}[shift={(0.000000cm,-0.000000cm)}]
\tikzset{every node/.style={circle,fill,draw=black,minimum size = 0.1cm,inner sep=0cm}}
\node (0) at (0.000000,1.000000) {};
\node (1) at (-0.781831,0.623490) {};
\node (2) at (-0.974928,-0.222521) {};
\node (3) at (-0.433884,-0.900969) {};
\node (4) at (0.433884,-0.900969) {};
\node (5) at (0.974928,-0.222521) {};
\node (6) at (0.781831,0.623490) {};
\path[-,every loop/.style={min distance=0.3cm}]
 (0) edge [in=130,out=50,loop] ()
 (2) edge (6)
 (3) edge [in=284,out=204,loop] ()
;\node[draw=none,fill=none] (a) at (0,-1.5cm) {\cmark};
\end{scope}
\begin{scope}[shift={(2.600000cm,-0.000000cm)}]
\tikzset{every node/.style={circle,fill,draw=black,minimum size = 0.1cm,inner sep=0cm}}
\node (0) at (0.000000,1.000000) {};
\node (1) at (-0.781831,0.623490) {};
\node (2) at (-0.974928,-0.222521) {};
\node (3) at (-0.433884,-0.900969) {};
\node (4) at (0.433884,-0.900969) {};
\node (5) at (0.974928,-0.222521) {};
\node (6) at (0.781831,0.623490) {};
\path[-,every loop/.style={min distance=0.3cm}]
 (1) edge (6)
 (2) edge (5)
 (3) edge (4)
;\node[draw=none,fill=none] (b) at (0,-1.5cm) {\cmark};
\end{scope}
\begin{scope}[shift={(5.200000cm,-0.000000cm)}]
\tikzset{every node/.style={circle,fill,draw=black,minimum size = 0.1cm,inner sep=0cm}}
\node (0) at (0.000000,1.000000) {};
\node (1) at (-0.781831,0.623490) {};
\node (2) at (-0.974928,-0.222521) {};
\node (3) at (-0.433884,-0.900969) {};
\node (4) at (0.433884,-0.900969) {};
\node (5) at (0.974928,-0.222521) {};
\node (6) at (0.781831,0.623490) {};
\path[-,every loop/.style={min distance=0.3cm}]
 (0) edge [in=130,out=50,loop] ()
 (1) edge (6)
 (2) edge (6)
;\node[draw=none,fill=none] (c) at (0,-1.5cm) {\xmark};
\end{scope}
\begin{scope}[shift={(7.800000cm,-0.000000cm)}]
\tikzset{every node/.style={circle,fill,draw=black,minimum size = 0.1cm,inner sep=0cm}}
\node (0) at (0.000000,1.000000) {};
\node (1) at (-0.781831,0.623490) {};
\node (2) at (-0.974928,-0.222521) {};
\node (3) at (-0.433884,-0.900969) {};
\node (4) at (0.433884,-0.900969) {};
\node (5) at (0.974928,-0.222521) {};
\node (6) at (0.781831,0.623490) {};
\path[-,every loop/.style={min distance=0.3cm}]
 (0) edge (1)
 (2) edge (3)
 (4) edge (5)
;\node[draw=none,fill=none] (d) at (0,-1.5cm) {\xmark};
\end{scope}
\begin{scope}[shift={(10.400000cm,-0.000000cm)}]
\tikzset{every node/.style={circle,fill,draw=black,minimum size = 0.1cm,inner sep=0cm}}
\node (0) at (0.000000,1.000000) {};
\node (1) at (-0.781831,0.623490) {};
\node (2) at (-0.974928,-0.222521) {};
\node (3) at (-0.433884,-0.900969) {};
\node (4) at (0.433884,-0.900969) {};
\node (5) at (0.974928,-0.222521) {};
\node (6) at (0.781831,0.623490) {};
\path[-,every loop/.style={min distance=0.3cm}]
 (1) edge (5)
 (2) edge (6)
 (3) edge (4)
;\node[draw=none,fill=none] (e) at (0,-1.5cm) {\xmark};
\end{scope}

%% file: plot4.tex
\begin{scope}[shift={(0cm,0cm)}]
\tikzset{every node/.style={circle,fill,draw=black,minimum size = 0.1cm,inner sep=0cm}}
\node (0) at (0.000000,1.000000) {};
\node (1) at (-1.000000,0.000000) {};
\node (2) at (0.000000,-1.000000) {};
\node (3) at (1.000000,0.000000) {};
\path[-,every loop/.style={min distance=0.3cm}]
 (0) edge (1)
 (0) edge (2)
 (0) edge (3)
 (1) edge (2)
 (1) edge (3)
 (2) edge [in=310,out=230,loop] ()
 (2) edge (3)
 (3) edge [in=400,out=320,loop] ()
;\end{scope}
\begin{scope}[shift={(0cm,-2.6cm)}]
\tikzset{every node/.style={circle,fill,draw=black,minimum size = 0.1cm,inner sep=0cm}}
\node (0) at (0.000000,1.000000) {};
\node (1) at (-1.000000,0.000000) {};
\node (2) at (0.000000,-1.000000) {};
\node (3) at (1.000000,0.000000) {};
\path[-,every loop/.style={min distance=0.3cm}]
 (0) edge (1)
 (0) edge (3)
 (1) edge [in=220,out=140,loop] ()
 (1) edge (2)
 (1) edge (3)
 (2) edge [in=310,out=230,loop] ()
 (2) edge (3)
 (3) edge [in=400,out=320,loop] ()
;\end{scope}
\begin{scope}[shift={(0cm,-5.2cm)}]
\tikzset{every node/.style={circle,fill,draw=black,minimum size = 0.1cm,inner sep=0cm}}
\node (0) at (0.000000,1.000000) {};
\node (1) at (-1.000000,0.000000) {};
\node (2) at (0.000000,-1.000000) {};
\node (3) at (1.000000,0.000000) {};
\path[-,every loop/.style={min distance=0.3cm}]
 (0) edge [in=130,out=50,loop] ()
 (0) edge (1)
 (0) edge (3)
 (1) edge [in=220,out=140,loop] ()
 (1) edge (2)
 (2) edge [in=310,out=230,loop] ()
 (2) edge (3)
 (3) edge [in=400,out=320,loop] ()
;\end{scope}

%% file: plot5.tex
\begin{scope}[shift={(0.000000cm,-0.000000cm)}]
\tikzset{every node/.style={circle,fill,draw=black,minimum size = 0.1cm,inner sep=0cm}}
\node (0) at (0.000000,1.000000) {};
\node (1) at (-0.951057,0.309017) {};
\node (2) at (-0.587785,-0.809017) {};
\node (3) at (0.587785,-0.809017) {};
\node (4) at (0.951057,0.309017) {};
\path[-,every loop/.style={min distance=0.3cm}]
 (0) edge (1)
 (0) edge (2)
 (0) edge (3)
 (0) edge (4)
 (1) edge (2)
 (1) edge (3)
 (1) edge (4)
 (2) edge (3)
 (2) edge (4)
 (3) edge (4)
;\end{scope}
\begin{scope}[shift={(2.600000cm,-0.000000cm)}]
\tikzset{every node/.style={circle,fill,draw=black,minimum size = 0.1cm,inner sep=0cm}}
\node (0) at (0.000000,1.000000) {};
\node (1) at (-0.951057,0.309017) {};
\node (2) at (-0.587785,-0.809017) {};
\node (3) at (0.587785,-0.809017) {};
\node (4) at (0.951057,0.309017) {};
\path[-,every loop/.style={min distance=0.3cm}]
 (0) edge (1)
 (0) edge (2)
 (0) edge (3)
 (0) edge (4)
 (1) edge (3)
 (1) edge (4)
 (2) edge (3)
 (2) edge (4)
 (3) edge (4)
 (4) edge [in=418,out=338,loop] ()
;\end{scope}
\begin{scope}[shift={(5.200000cm,-0.000000cm)}]
\tikzset{every node/.style={circle,fill,draw=black,minimum size = 0.1cm,inner sep=0cm}}
\node (0) at (0.000000,1.000000) {};
\node (1) at (-0.951057,0.309017) {};
\node (2) at (-0.587785,-0.809017) {};
\node (3) at (0.587785,-0.809017) {};
\node (4) at (0.951057,0.309017) {};
\path[-,every loop/.style={min distance=0.3cm}]
 (0) edge (1)
 (0) edge (2)
 (0) edge (4)
 (1) edge (2)
 (1) edge (3)
 (1) edge (4)
 (2) edge (3)
 (2) edge (4)
 (3) edge (4)
 (4) edge [in=418,out=338,loop] ()
;\end{scope}
\begin{scope}[shift={(7.800000cm,-0.000000cm)}]
\tikzset{every node/.style={circle,fill,draw=black,minimum size = 0.1cm,inner sep=0cm}}
\node (0) at (0.000000,1.000000) {};
\node (1) at (-0.951057,0.309017) {};
\node (2) at (-0.587785,-0.809017) {};
\node (3) at (0.587785,-0.809017) {};
\node (4) at (0.951057,0.309017) {};
\path[-,every loop/.style={min distance=0.3cm}]
 (0) edge (2)
 (0) edge (3)
 (0) edge (4)
 (1) edge (3)
 (1) edge (4)
 (2) edge (3)
 (2) edge (4)
 (3) edge [in=346,out=266,loop] ()
 (3) edge (4)
 (4) edge [in=418,out=338,loop] ()
;\end{scope}
\begin{scope}[shift={(0.000000cm,-2.600000cm)}]
\tikzset{every node/.style={circle,fill,draw=black,minimum size = 0.1cm,inner sep=0cm}}
\node (0) at (0.000000,1.000000) {};
\node (1) at (-0.951057,0.309017) {};
\node (2) at (-0.587785,-0.809017) {};
\node (3) at (0.587785,-0.809017) {};
\node (4) at (0.951057,0.309017) {};
\path[-,every loop/.style={min distance=0.3cm}]
 (0) edge (2)
 (0) edge (4)
 (1) edge (2)
 (1) edge (3)
 (1) edge (4)
 (2) edge (3)
 (2) edge (4)
 (3) edge [in=346,out=266,loop] ()
 (3) edge (4)
 (4) edge [in=418,out=338,loop] ()
;\end{scope}
\begin{scope}[shift={(2.600000cm,-2.600000cm)}]
\tikzset{every node/.style={circle,fill,draw=black,minimum size = 0.1cm,inner sep=0cm}}
\node (0) at (0.000000,1.000000) {};
\node (1) at (-0.951057,0.309017) {};
\node (2) at (-0.587785,-0.809017) {};
\node (3) at (0.587785,-0.809017) {};
\node (4) at (0.951057,0.309017) {};
\path[-,every loop/.style={min distance=0.3cm}]
 (0) edge (1)
 (0) edge (2)
 (0) edge (4)
 (1) edge (2)
 (1) edge (3)
 (1) edge (4)
 (2) edge (3)
 (3) edge [in=346,out=266,loop] ()
 (3) edge (4)
 (4) edge [in=418,out=338,loop] ()
;\end{scope}
\begin{scope}[shift={(5.200000cm,-2.600000cm)}]
\tikzset{every node/.style={circle,fill,draw=black,minimum size = 0.1cm,inner sep=0cm}}
\node (0) at (0.000000,1.000000) {};
\node (1) at (-0.951057,0.309017) {};
\node (2) at (-0.587785,-0.809017) {};
\node (3) at (0.587785,-0.809017) {};
\node (4) at (0.951057,0.309017) {};
\path[-,every loop/.style={min distance=0.3cm}]
 (0) edge (1)
 (0) edge (2)
 (0) edge (4)
 (1) edge (2)
 (1) edge (4)
 (2) edge [in=274,out=194,loop] ()
 (2) edge (3)
 (2) edge (4)
 (3) edge (4)
 (4) edge [in=418,out=338,loop] ()
;\end{scope}
\begin{scope}[shift={(7.800000cm,-2.600000cm)}]
\tikzset{every node/.style={circle,fill,draw=black,minimum size = 0.1cm,inner sep=0cm}}
\node (0) at (0.000000,1.000000) {};
\node (1) at (-0.951057,0.309017) {};
\node (2) at (-0.587785,-0.809017) {};
\node (3) at (0.587785,-0.809017) {};
\node (4) at (0.951057,0.309017) {};
\path[-,every loop/.style={min distance=0.3cm}]
 (0) edge (2)
 (0) edge (4)
 (1) edge (2)
 (1) edge (4)
 (2) edge [in=274,out=194,loop] ()
 (2) edge (3)
 (2) edge (4)
 (3) edge [in=346,out=266,loop] ()
 (3) edge (4)
 (4) edge [in=418,out=338,loop] ()
;\end{scope}
\begin{scope}[shift={(0.000000cm,-5.200000cm)}]
\tikzset{every node/.style={circle,fill,draw=black,minimum size = 0.1cm,inner sep=0cm}}
\node (0) at (0.000000,1.000000) {};
\node (1) at (-0.951057,0.309017) {};
\node (2) at (-0.587785,-0.809017) {};
\node (3) at (0.587785,-0.809017) {};
\node (4) at (0.951057,0.309017) {};
\path[-,every loop/.style={min distance=0.3cm}]
 (0) edge (1)
 (0) edge (2)
 (0) edge (4)
 (1) edge (2)
 (1) edge (4)
 (2) edge [in=274,out=194,loop] ()
 (2) edge (3)
 (3) edge [in=346,out=266,loop] ()
 (3) edge (4)
 (4) edge [in=418,out=338,loop] ()
;\end{scope}
\begin{scope}[shift={(2.600000cm,-5.200000cm)}]
\tikzset{every node/.style={circle,fill,draw=black,minimum size = 0.1cm,inner sep=0cm}}
\node (0) at (0.000000,1.000000) {};
\node (1) at (-0.951057,0.309017) {};
\node (2) at (-0.587785,-0.809017) {};
\node (3) at (0.587785,-0.809017) {};
\node (4) at (0.951057,0.309017) {};
\path[-,every loop/.style={min distance=0.3cm}]
 (0) edge (1)
 (0) edge (4)
 (1) edge [in=202,out=122,loop] ()
 (1) edge (2)
 (1) edge (3)
 (1) edge (4)
 (2) edge (3)
 (3) edge [in=346,out=266,loop] ()
 (3) edge (4)
 (4) edge [in=418,out=338,loop] ()
;\end{scope}
\begin{scope}[shift={(5.200000cm,-5.200000cm)}]
\tikzset{every node/.style={circle,fill,draw=black,minimum size = 0.1cm,inner sep=0cm}}
\node (0) at (0.000000,1.000000) {};
\node (1) at (-0.951057,0.309017) {};
\node (2) at (-0.587785,-0.809017) {};
\node (3) at (0.587785,-0.809017) {};
\node (4) at (0.951057,0.309017) {};
\path[-,every loop/.style={min distance=0.3cm}]
 (0) edge (1)
 (0) edge (4)
 (1) edge [in=202,out=122,loop] ()
 (1) edge (2)
 (1) edge (4)
 (2) edge [in=274,out=194,loop] ()
 (2) edge (3)
 (3) edge [in=346,out=266,loop] ()
 (3) edge (4)
 (4) edge [in=418,out=338,loop] ()
;\end{scope}
\begin{scope}[shift={(7.800000cm,-5.200000cm)}]
\tikzset{every node/.style={circle,fill,draw=black,minimum size = 0.1cm,inner sep=0cm}}
\node (0) at (0.000000,1.000000) {};
\node (1) at (-0.951057,0.309017) {};
\node (2) at (-0.587785,-0.809017) {};
\node (3) at (0.587785,-0.809017) {};
\node (4) at (0.951057,0.309017) {};
\path[-,every loop/.style={min distance=0.3cm}]
 (0) edge [in=130,out=50,loop] ()
 (0) edge (1)
 (0) edge (4)
 (1) edge [in=202,out=122,loop] ()
 (1) edge (2)
 (2) edge [in=274,out=194,loop] ()
 (2) edge (3)
 (3) edge [in=346,out=266,loop] ()
 (3) edge (4)
 (4) edge [in=418,out=338,loop] ()
;\end{scope}

%% file: plot4edge.tex
\begin{scope}[shift={(0.000000cm,-0.000000cm)}]
\begin{scope}[shift={(0.000000cm,-0.000000cm)}]
\tikzset{every node/.style={circle,fill,draw=black,minimum size = 0.1cm,inner sep=0cm}}
\node (0) at (0.000000,1.000000) {};
\node (1) at (-1.000000,0.000000) {};
\node (2) at (0.000000,-1.000000) {};
\node (3) at (1.000000,0.000000) {};
\path[-,every loop/.style={min distance=0.3cm}]
 (0) edge (1)
 (0) edge (2)
 (0) edge (3)
 (1) edge (2)
 (1) edge (3)
 (2) edge [in=310,out=230,loop] ()
 (2) edge (3)
 (3) edge [in=400,out=320,loop] ()
;\end{scope}
\begin{scope}[shift={(2.500000cm,-0.000000cm)}]
\tikzset{every node/.style={circle,fill,draw=black,inner sep=0cm}}
\node[minimum size=0.04cm] (0-0) at (-0.600000cm,0.600000cm) {};
\node[minimum size=0.04cm] (0-1) at (-0.300000cm,0.300000cm) {};
\node[fill=none,minimum size=0.1cm] (0-1-n) at (-0.300000cm,0.300000cm) {};
\node[minimum size=0.04cm] (0-2) at (0.000000cm,0.000000cm) {};
\node[fill=none,minimum size=0.1cm] (0-2-n) at (0.000000cm,0.000000cm) {};
\node[minimum size=0.04cm] (0-3) at (0.300000cm,-0.300000cm) {};
\node[fill=none,minimum size=0.1cm] (0-3-n) at (0.300000cm,-0.300000cm) {};
\node[draw=none,fill=none] (0-3-outside) at (-0.900000cm,0.300000cm) {};
\node[minimum size=0.04cm] (1-1) at (0.000000cm,0.600000cm) {};
\node[minimum size=0.04cm] (1-2) at (0.300000cm,0.300000cm) {};
\node[fill=none,minimum size=0.1cm] (1-2-n) at (0.300000cm,0.300000cm) {};
\node[draw=none,fill=none] (1-2-outside) at (-0.900000cm,-0.300000cm) {};
\node[minimum size=0.04cm] (1-3) at (-0.600000cm,0.000000cm) {};
\node[fill=none,minimum size=0.1cm] (1-3-n) at (-0.600000cm,0.000000cm) {};
\node[draw=none,fill=none] (1-3-outside) at (0.600000cm,0.000000cm) {};
\node[minimum size=0.04cm] (2-2) at (-0.600000cm,-0.600000cm) {};
\node[fill=none,minimum size=0.1cm] (2-2-n) at (-0.600000cm,-0.600000cm) {};
\node[draw=none,fill=none] (2-2-outside) at (0.600000cm,0.600000cm) {};
\node[minimum size=0.04cm] (2-3) at (-0.300000cm,-0.300000cm) {};
\node[fill=none,minimum size=0.1cm] (2-3-n) at (-0.300000cm,-0.300000cm) {};
\node[minimum size=0.04cm] (3-3) at (0.000000cm,-0.600000cm) {};
\node[fill=none,minimum size=0.1cm] (3-3-n) at (0.000000cm,-0.600000cm) {};
\draw[-] (2-3-n) edge (3-3-n);
\draw[-] (0-1-n) edge (0-2-n);
\draw[-] (3-3-n) edge (0-3-n);
\draw[-] (0-2-n) edge (1-2-n);
\draw[-] (0-3-n) edge ($(0-3-n)!0.5!(1-3-outside)$);
\draw[-] (1-3-n) edge ($(1-3-n)!0.5!(0-3-outside)$);
\draw[-] (1-2-n) edge ($(1-2-n)!0.5!(2-2-outside)$);
\draw[-] (2-2-n) edge ($(2-2-n)!0.5!(1-2-outside)$);
\draw[-] (2-2-n) edge (2-3-n);
\draw[-] (1-3-n) edge (0-1-n);
;
\end{scope}
\end{scope}
\begin{scope}[shift={(5.000000cm,-0.000000cm)}]
\begin{scope}[shift={(0.000000cm,-0.000000cm)}]
\tikzset{every node/.style={circle,fill,draw=black,minimum size = 0.1cm,inner sep=0cm}}
\node (0) at (0.000000,1.000000) {};
\node (1) at (-1.000000,0.000000) {};
\node (2) at (0.000000,-1.000000) {};
\node (3) at (1.000000,0.000000) {};
\path[-,every loop/.style={min distance=0.3cm}]
 (0) edge (1)
 (0) edge (3)
 (1) edge [in=220,out=140,loop] ()
 (1) edge (2)
 (1) edge (3)
 (2) edge [in=310,out=230,loop] ()
 (2) edge (3)
 (3) edge [in=400,out=320,loop] ()
;\end{scope}
\begin{scope}[shift={(2.500000cm,-0.000000cm)}]
\tikzset{every node/.style={circle,fill,draw=black,inner sep=0cm}}
\node[minimum size=0.04cm] (0-0) at (-0.600000cm,0.600000cm) {};
\node[minimum size=0.04cm] (0-1) at (-0.300000cm,0.300000cm) {};
\node[fill=none,minimum size=0.1cm] (0-1-n) at (-0.300000cm,0.300000cm) {};
\node[minimum size=0.04cm] (0-2) at (0.000000cm,0.000000cm) {};
\node[minimum size=0.04cm] (0-3) at (0.300000cm,-0.300000cm) {};
\node[fill=none,minimum size=0.1cm] (0-3-n) at (0.300000cm,-0.300000cm) {};
\node[draw=none,fill=none] (0-3-outside) at (-0.900000cm,0.300000cm) {};
\node[minimum size=0.04cm] (1-1) at (0.000000cm,0.600000cm) {};
\node[fill=none,minimum size=0.1cm] (1-1-n) at (0.000000cm,0.600000cm) {};
\node[minimum size=0.04cm] (1-2) at (0.300000cm,0.300000cm) {};
\node[fill=none,minimum size=0.1cm] (1-2-n) at (0.300000cm,0.300000cm) {};
\node[draw=none,fill=none] (1-2-outside) at (-0.900000cm,-0.300000cm) {};
\node[minimum size=0.04cm] (1-3) at (-0.600000cm,0.000000cm) {};
\node[fill=none,minimum size=0.1cm] (1-3-n) at (-0.600000cm,0.000000cm) {};
\node[draw=none,fill=none] (1-3-outside) at (0.600000cm,0.000000cm) {};
\node[minimum size=0.04cm] (2-2) at (-0.600000cm,-0.600000cm) {};
\node[fill=none,minimum size=0.1cm] (2-2-n) at (-0.600000cm,-0.600000cm) {};
\node[draw=none,fill=none] (2-2-outside) at (0.600000cm,0.600000cm) {};
\node[minimum size=0.04cm] (2-3) at (-0.300000cm,-0.300000cm) {};
\node[fill=none,minimum size=0.1cm] (2-3-n) at (-0.300000cm,-0.300000cm) {};
\node[minimum size=0.04cm] (3-3) at (0.000000cm,-0.600000cm) {};
\node[fill=none,minimum size=0.1cm] (3-3-n) at (0.000000cm,-0.600000cm) {};
\draw[-] (2-3-n) edge (3-3-n);
\draw[-] (0-1-n) edge (1-1-n);
\draw[-] (0-3-n) edge ($(0-3-n)!0.5!(1-3-outside)$);
\draw[-] (1-3-n) edge ($(1-3-n)!0.5!(0-3-outside)$);
\draw[-] (1-2-n) edge ($(1-2-n)!0.5!(2-2-outside)$);
\draw[-] (2-2-n) edge ($(2-2-n)!0.5!(1-2-outside)$);
\draw[-] (3-3-n) edge (0-3-n);
\draw[-] (1-1-n) edge (1-2-n);
\draw[-] (2-2-n) edge (2-3-n);
\draw[-] (1-3-n) edge (0-1-n);
;
\end{scope}
\end{scope}
\begin{scope}[shift={(10.000000cm,-0.000000cm)}]
\begin{scope}[shift={(0.000000cm,-0.000000cm)}]
\tikzset{every node/.style={circle,fill,draw=black,minimum size = 0.1cm,inner sep=0cm}}
\node (0) at (0.000000,1.000000) {};
\node (1) at (-1.000000,0.000000) {};
\node (2) at (0.000000,-1.000000) {};
\node (3) at (1.000000,0.000000) {};
\path[-,every loop/.style={min distance=0.3cm}]
 (0) edge [in=130,out=50,loop] ()
 (0) edge (1)
 (0) edge (3)
 (1) edge [in=220,out=140,loop] ()
 (1) edge (2)
 (2) edge [in=310,out=230,loop] ()
 (2) edge (3)
 (3) edge [in=400,out=320,loop] ()
;\end{scope}
\begin{scope}[shift={(2.500000cm,-0.000000cm)}]
\tikzset{every node/.style={circle,fill,draw=black,inner sep=0cm}}
\node[minimum size=0.04cm] (0-0) at (-0.600000cm,0.600000cm) {};
\node[fill=none,minimum size=0.1cm] (0-0-n) at (-0.600000cm,0.600000cm) {};
\node[draw=none,fill=none] (0-0-outside) at (0.600000cm,-0.600000cm) {};
\node[minimum size=0.04cm] (0-1) at (-0.300000cm,0.300000cm) {};
\node[fill=none,minimum size=0.1cm] (0-1-n) at (-0.300000cm,0.300000cm) {};
\node[minimum size=0.04cm] (0-2) at (0.000000cm,0.000000cm) {};
\node[minimum size=0.04cm] (0-3) at (0.300000cm,-0.300000cm) {};
\node[fill=none,minimum size=0.1cm] (0-3-n) at (0.300000cm,-0.300000cm) {};
\node[draw=none,fill=none] (0-3-outside) at (-0.900000cm,0.300000cm) {};
\node[minimum size=0.04cm] (1-1) at (0.000000cm,0.600000cm) {};
\node[fill=none,minimum size=0.1cm] (1-1-n) at (0.000000cm,0.600000cm) {};
\node[minimum size=0.04cm] (1-2) at (0.300000cm,0.300000cm) {};
\node[fill=none,minimum size=0.1cm] (1-2-n) at (0.300000cm,0.300000cm) {};
\node[draw=none,fill=none] (1-2-outside) at (-0.900000cm,-0.300000cm) {};
\node[minimum size=0.04cm] (1-3) at (-0.600000cm,0.000000cm) {};
\node[minimum size=0.04cm] (2-2) at (-0.600000cm,-0.600000cm) {};
\node[fill=none,minimum size=0.1cm] (2-2-n) at (-0.600000cm,-0.600000cm) {};
\node[draw=none,fill=none] (2-2-outside) at (0.600000cm,0.600000cm) {};
\node[minimum size=0.04cm] (2-3) at (-0.300000cm,-0.300000cm) {};
\node[fill=none,minimum size=0.1cm] (2-3-n) at (-0.300000cm,-0.300000cm) {};
\node[minimum size=0.04cm] (3-3) at (0.000000cm,-0.600000cm) {};
\node[fill=none,minimum size=0.1cm] (3-3-n) at (0.000000cm,-0.600000cm) {};
\draw[-] (2-2-n) edge (2-3-n);
\draw[-] (0-0-n) edge (0-1-n);
\draw[-] (2-3-n) edge (3-3-n);
\draw[-] (0-1-n) edge (1-1-n);
\draw[-] (0-3-n) edge ($(0-3-n)!0.5!(0-0-outside)$);
\draw[-] (0-0-n) edge ($(0-0-n)!0.5!(0-3-outside)$);
\draw[-] (1-2-n) edge ($(1-2-n)!0.5!(2-2-outside)$);
\draw[-] (2-2-n) edge ($(2-2-n)!0.5!(1-2-outside)$);
\draw[-] (3-3-n) edge (0-3-n);
\draw[-] (1-1-n) edge (1-2-n);
;
\end{scope}
\end{scope}

%% file: plot5edge.tex
\begin{scope}[shift={(0.000000cm,-0.000000cm)}]
\begin{scope}[shift={(0.000000cm,-0.000000cm)}]
\tikzset{every node/.style={circle,fill,draw=black,minimum size = 0.1cm,inner sep=0cm}}
\node (0) at (0.000000,1.000000) {};
\node (1) at (-0.951057,0.309017) {};
\node (2) at (-0.587785,-0.809017) {};
\node (3) at (0.587785,-0.809017) {};
\node (4) at (0.951057,0.309017) {};
\path[-,every loop/.style={min distance=0.3cm}]
 (0) edge (1)
 (0) edge (2)
 (0) edge (3)
 (0) edge (4)
 (1) edge (2)
 (1) edge (3)
 (1) edge (4)
 (2) edge (3)
 (2) edge (4)
 (3) edge (4)
;\end{scope}
\begin{scope}[shift={(2.500000cm,-0.000000cm)}]
\tikzset{every node/.style={circle,fill,draw=black,inner sep=0cm}}
\node[minimum size=0.04cm] (0-0) at (-0.750000cm,0.750000cm) {};
\node[minimum size=0.04cm] (0-1) at (-0.450000cm,0.450000cm) {};
\node[fill=none,minimum size=0.1cm] (0-1-n) at (-0.450000cm,0.450000cm) {};
\node[minimum size=0.04cm] (0-2) at (-0.150000cm,0.150000cm) {};
\node[fill=none,minimum size=0.1cm] (0-2-n) at (-0.150000cm,0.150000cm) {};
\node[minimum size=0.04cm] (0-3) at (0.150000cm,-0.150000cm) {};
\node[fill=none,minimum size=0.1cm] (0-3-n) at (0.150000cm,-0.150000cm) {};
\node[minimum size=0.04cm] (0-4) at (0.450000cm,-0.450000cm) {};
\node[fill=none,minimum size=0.1cm] (0-4-n) at (0.450000cm,-0.450000cm) {};
\node[draw=none,fill=none] (0-4-outside) at (-1.050000cm,0.450000cm) {};
\node[minimum size=0.04cm] (1-1) at (-0.150000cm,0.750000cm) {};
\node[minimum size=0.04cm] (1-2) at (0.150000cm,0.450000cm) {};
\node[fill=none,minimum size=0.1cm] (1-2-n) at (0.150000cm,0.450000cm) {};
\node[minimum size=0.04cm] (1-3) at (0.450000cm,0.150000cm) {};
\node[fill=none,minimum size=0.1cm] (1-3-n) at (0.450000cm,0.150000cm) {};
\node[draw=none,fill=none] (1-3-outside) at (-1.050000cm,-0.150000cm) {};
\node[minimum size=0.04cm] (1-4) at (-0.750000cm,0.150000cm) {};
\node[fill=none,minimum size=0.1cm] (1-4-n) at (-0.750000cm,0.150000cm) {};
\node[draw=none,fill=none] (1-4-outside) at (0.750000cm,-0.150000cm) {};
\node[minimum size=0.04cm] (2-2) at (0.450000cm,0.750000cm) {};
\node[minimum size=0.04cm] (2-3) at (-0.750000cm,-0.450000cm) {};
\node[fill=none,minimum size=0.1cm] (2-3-n) at (-0.750000cm,-0.450000cm) {};
\node[draw=none,fill=none] (2-3-outside) at (0.750000cm,0.450000cm) {};
\node[minimum size=0.04cm] (2-4) at (-0.450000cm,-0.150000cm) {};
\node[fill=none,minimum size=0.1cm] (2-4-n) at (-0.450000cm,-0.150000cm) {};
\node[minimum size=0.04cm] (3-3) at (-0.450000cm,-0.750000cm) {};
\node[minimum size=0.04cm] (3-4) at (-0.150000cm,-0.450000cm) {};
\node[fill=none,minimum size=0.1cm] (3-4-n) at (-0.150000cm,-0.450000cm) {};
\node[minimum size=0.04cm] (4-4) at (0.150000cm,-0.750000cm) {};
\draw[-] (2-4-n) edge (3-4-n);
\draw[-] (0-1-n) edge (0-2-n);
\draw[-] (3-4-n) edge (0-3-n);
\draw[-] (0-2-n) edge (1-2-n);
\draw[-] (0-3-n) edge (0-4-n);
\draw[-] (1-2-n) edge (1-3-n);
\draw[-] (0-4-n) edge ($(0-4-n)!0.5!(1-4-outside)$);
\draw[-] (1-4-n) edge ($(1-4-n)!0.5!(0-4-outside)$);
\draw[-] (1-3-n) edge ($(1-3-n)!0.5!(2-3-outside)$);
\draw[-] (2-3-n) edge ($(2-3-n)!0.5!(1-3-outside)$);
\draw[-] (2-3-n) edge (2-4-n);
\draw[-] (1-4-n) edge (0-1-n);
;
\end{scope}
\end{scope}
\begin{scope}[shift={(5.000000cm,-0.000000cm)}]
\begin{scope}[shift={(0.000000cm,-0.000000cm)}]
\tikzset{every node/.style={circle,fill,draw=black,minimum size = 0.1cm,inner sep=0cm}}
\node (0) at (0.000000,1.000000) {};
\node (1) at (-0.951057,0.309017) {};
\node (2) at (-0.587785,-0.809017) {};
\node (3) at (0.587785,-0.809017) {};
\node (4) at (0.951057,0.309017) {};
\path[-,every loop/.style={min distance=0.3cm}]
 (0) edge (1)
 (0) edge (2)
 (0) edge (3)
 (0) edge (4)
 (1) edge (3)
 (1) edge (4)
 (2) edge (3)
 (2) edge (4)
 (3) edge (4)
 (4) edge [in=418,out=338,loop] ()
;\end{scope}
\begin{scope}[shift={(2.500000cm,-0.000000cm)}]
\tikzset{every node/.style={circle,fill,draw=black,inner sep=0cm}}
\node[minimum size=0.04cm] (0-0) at (-0.750000cm,0.750000cm) {};
\node[minimum size=0.04cm] (0-1) at (-0.450000cm,0.450000cm) {};
\node[fill=none,minimum size=0.1cm] (0-1-n) at (-0.450000cm,0.450000cm) {};
\node[minimum size=0.04cm] (0-2) at (-0.150000cm,0.150000cm) {};
\node[fill=none,minimum size=0.1cm] (0-2-n) at (-0.150000cm,0.150000cm) {};
\node[minimum size=0.04cm] (0-3) at (0.150000cm,-0.150000cm) {};
\node[fill=none,minimum size=0.1cm] (0-3-n) at (0.150000cm,-0.150000cm) {};
\node[minimum size=0.04cm] (0-4) at (0.450000cm,-0.450000cm) {};
\node[fill=none,minimum size=0.1cm] (0-4-n) at (0.450000cm,-0.450000cm) {};
\node[draw=none,fill=none] (0-4-outside) at (-1.050000cm,0.450000cm) {};
\node[minimum size=0.04cm] (1-1) at (-0.150000cm,0.750000cm) {};
\node[minimum size=0.04cm] (1-2) at (0.150000cm,0.450000cm) {};
\node[minimum size=0.04cm] (1-3) at (0.450000cm,0.150000cm) {};
\node[fill=none,minimum size=0.1cm] (1-3-n) at (0.450000cm,0.150000cm) {};
\node[draw=none,fill=none] (1-3-outside) at (-1.050000cm,-0.150000cm) {};
\node[minimum size=0.04cm] (1-4) at (-0.750000cm,0.150000cm) {};
\node[fill=none,minimum size=0.1cm] (1-4-n) at (-0.750000cm,0.150000cm) {};
\node[draw=none,fill=none] (1-4-outside) at (0.750000cm,-0.150000cm) {};
\node[minimum size=0.04cm] (2-2) at (0.450000cm,0.750000cm) {};
\node[minimum size=0.04cm] (2-3) at (-0.750000cm,-0.450000cm) {};
\node[fill=none,minimum size=0.1cm] (2-3-n) at (-0.750000cm,-0.450000cm) {};
\node[draw=none,fill=none] (2-3-outside) at (0.750000cm,0.450000cm) {};
\node[minimum size=0.04cm] (2-4) at (-0.450000cm,-0.150000cm) {};
\node[fill=none,minimum size=0.1cm] (2-4-n) at (-0.450000cm,-0.150000cm) {};
\node[minimum size=0.04cm] (3-3) at (-0.450000cm,-0.750000cm) {};
\node[minimum size=0.04cm] (3-4) at (-0.150000cm,-0.450000cm) {};
\node[fill=none,minimum size=0.1cm] (3-4-n) at (-0.150000cm,-0.450000cm) {};
\node[minimum size=0.04cm] (4-4) at (0.150000cm,-0.750000cm) {};
\node[fill=none,minimum size=0.1cm] (4-4-n) at (0.150000cm,-0.750000cm) {};
\draw[-] (2-4-n) edge (3-4-n);
\draw[-] (0-1-n) edge (0-2-n);
\draw[-] (3-4-n) edge (4-4-n);
\draw[-] (0-2-n) edge (0-3-n);
\draw[-] (4-4-n) edge (0-4-n);
\draw[-] (0-3-n) edge (1-3-n);
\draw[-] (0-4-n) edge ($(0-4-n)!0.5!(1-4-outside)$);
\draw[-] (1-4-n) edge ($(1-4-n)!0.5!(0-4-outside)$);
\draw[-] (1-3-n) edge ($(1-3-n)!0.5!(2-3-outside)$);
\draw[-] (2-3-n) edge ($(2-3-n)!0.5!(1-3-outside)$);
\draw[-] (2-3-n) edge (2-4-n);
\draw[-] (1-4-n) edge (0-1-n);
;
\end{scope}
\end{scope}
\begin{scope}[shift={(10.000000cm,-0.000000cm)}]
\begin{scope}[shift={(0.000000cm,-0.000000cm)}]
\tikzset{every node/.style={circle,fill,draw=black,minimum size = 0.1cm,inner sep=0cm}}
\node (0) at (0.000000,1.000000) {};
\node (1) at (-0.951057,0.309017) {};
\node (2) at (-0.587785,-0.809017) {};
\node (3) at (0.587785,-0.809017) {};
\node (4) at (0.951057,0.309017) {};
\path[-,every loop/.style={min distance=0.3cm}]
 (0) edge (1)
 (0) edge (2)
 (0) edge (4)
 (1) edge (2)
 (1) edge (3)
 (1) edge (4)
 (2) edge (3)
 (2) edge (4)
 (3) edge (4)
 (4) edge [in=418,out=338,loop] ()
;\end{scope}
\begin{scope}[shift={(2.500000cm,-0.000000cm)}]
\tikzset{every node/.style={circle,fill,draw=black,inner sep=0cm}}
\node[minimum size=0.04cm] (0-0) at (-0.750000cm,0.750000cm) {};
\node[minimum size=0.04cm] (0-1) at (-0.450000cm,0.450000cm) {};
\node[fill=none,minimum size=0.1cm] (0-1-n) at (-0.450000cm,0.450000cm) {};
\node[minimum size=0.04cm] (0-2) at (-0.150000cm,0.150000cm) {};
\node[fill=none,minimum size=0.1cm] (0-2-n) at (-0.150000cm,0.150000cm) {};
\node[minimum size=0.04cm] (0-3) at (0.150000cm,-0.150000cm) {};
\node[minimum size=0.04cm] (0-4) at (0.450000cm,-0.450000cm) {};
\node[fill=none,minimum size=0.1cm] (0-4-n) at (0.450000cm,-0.450000cm) {};
\node[draw=none,fill=none] (0-4-outside) at (-1.050000cm,0.450000cm) {};
\node[minimum size=0.04cm] (1-1) at (-0.150000cm,0.750000cm) {};
\node[minimum size=0.04cm] (1-2) at (0.150000cm,0.450000cm) {};
\node[fill=none,minimum size=0.1cm] (1-2-n) at (0.150000cm,0.450000cm) {};
\node[minimum size=0.04cm] (1-3) at (0.450000cm,0.150000cm) {};
\node[fill=none,minimum size=0.1cm] (1-3-n) at (0.450000cm,0.150000cm) {};
\node[draw=none,fill=none] (1-3-outside) at (-1.050000cm,-0.150000cm) {};
\node[minimum size=0.04cm] (1-4) at (-0.750000cm,0.150000cm) {};
\node[fill=none,minimum size=0.1cm] (1-4-n) at (-0.750000cm,0.150000cm) {};
\node[draw=none,fill=none] (1-4-outside) at (0.750000cm,-0.150000cm) {};
\node[minimum size=0.04cm] (2-2) at (0.450000cm,0.750000cm) {};
\node[minimum size=0.04cm] (2-3) at (-0.750000cm,-0.450000cm) {};
\node[fill=none,minimum size=0.1cm] (2-3-n) at (-0.750000cm,-0.450000cm) {};
\node[draw=none,fill=none] (2-3-outside) at (0.750000cm,0.450000cm) {};
\node[minimum size=0.04cm] (2-4) at (-0.450000cm,-0.150000cm) {};
\node[fill=none,minimum size=0.1cm] (2-4-n) at (-0.450000cm,-0.150000cm) {};
\node[minimum size=0.04cm] (3-3) at (-0.450000cm,-0.750000cm) {};
\node[minimum size=0.04cm] (3-4) at (-0.150000cm,-0.450000cm) {};
\node[fill=none,minimum size=0.1cm] (3-4-n) at (-0.150000cm,-0.450000cm) {};
\node[minimum size=0.04cm] (4-4) at (0.150000cm,-0.750000cm) {};
\node[fill=none,minimum size=0.1cm] (4-4-n) at (0.150000cm,-0.750000cm) {};
\draw[-] (2-4-n) edge (3-4-n);
\draw[-] (0-1-n) edge (0-2-n);
\draw[-] (3-4-n) edge (4-4-n);
\draw[-] (0-2-n) edge (1-2-n);
\draw[-] (0-4-n) edge ($(0-4-n)!0.5!(1-4-outside)$);
\draw[-] (1-4-n) edge ($(1-4-n)!0.5!(0-4-outside)$);
\draw[-] (1-3-n) edge ($(1-3-n)!0.5!(2-3-outside)$);
\draw[-] (2-3-n) edge ($(2-3-n)!0.5!(1-3-outside)$);
\draw[-] (4-4-n) edge (0-4-n);
\draw[-] (1-2-n) edge (1-3-n);
\draw[-] (2-3-n) edge (2-4-n);
\draw[-] (1-4-n) edge (0-1-n);
;
\end{scope}
\end{scope}
\begin{scope}[shift={(0.000000cm,-2.500000cm)}]
\begin{scope}[shift={(0.000000cm,-0.000000cm)}]
\tikzset{every node/.style={circle,fill,draw=black,minimum size = 0.1cm,inner sep=0cm}}
\node (0) at (0.000000,1.000000) {};
\node (1) at (-0.951057,0.309017) {};
\node (2) at (-0.587785,-0.809017) {};
\node (3) at (0.587785,-0.809017) {};
\node (4) at (0.951057,0.309017) {};
\path[-,every loop/.style={min distance=0.3cm}]
 (0) edge (2)
 (0) edge (3)
 (0) edge (4)
 (1) edge (3)
 (1) edge (4)
 (2) edge (3)
 (2) edge (4)
 (3) edge [in=346,out=266,loop] ()
 (3) edge (4)
 (4) edge [in=418,out=338,loop] ()
;\end{scope}
\begin{scope}[shift={(2.500000cm,-0.000000cm)}]
\tikzset{every node/.style={circle,fill,draw=black,inner sep=0cm}}
\node[minimum size=0.04cm] (0-0) at (-0.750000cm,0.750000cm) {};
\node[minimum size=0.04cm] (0-1) at (-0.450000cm,0.450000cm) {};
\node[minimum size=0.04cm] (0-2) at (-0.150000cm,0.150000cm) {};
\node[fill=none,minimum size=0.1cm] (0-2-n) at (-0.150000cm,0.150000cm) {};
\node[minimum size=0.04cm] (0-3) at (0.150000cm,-0.150000cm) {};
\node[fill=none,minimum size=0.1cm] (0-3-n) at (0.150000cm,-0.150000cm) {};
\node[minimum size=0.04cm] (0-4) at (0.450000cm,-0.450000cm) {};
\node[fill=none,minimum size=0.1cm] (0-4-n) at (0.450000cm,-0.450000cm) {};
\node[draw=none,fill=none] (0-4-outside) at (-1.050000cm,0.450000cm) {};
\node[minimum size=0.04cm] (1-1) at (-0.150000cm,0.750000cm) {};
\node[minimum size=0.04cm] (1-2) at (0.150000cm,0.450000cm) {};
\node[minimum size=0.04cm] (1-3) at (0.450000cm,0.150000cm) {};
\node[fill=none,minimum size=0.1cm] (1-3-n) at (0.450000cm,0.150000cm) {};
\node[draw=none,fill=none] (1-3-outside) at (-1.050000cm,-0.150000cm) {};
\node[minimum size=0.04cm] (1-4) at (-0.750000cm,0.150000cm) {};
\node[fill=none,minimum size=0.1cm] (1-4-n) at (-0.750000cm,0.150000cm) {};
\node[draw=none,fill=none] (1-4-outside) at (0.750000cm,-0.150000cm) {};
\node[minimum size=0.04cm] (2-2) at (0.450000cm,0.750000cm) {};
\node[minimum size=0.04cm] (2-3) at (-0.750000cm,-0.450000cm) {};
\node[fill=none,minimum size=0.1cm] (2-3-n) at (-0.750000cm,-0.450000cm) {};
\node[draw=none,fill=none] (2-3-outside) at (0.750000cm,0.450000cm) {};
\node[minimum size=0.04cm] (2-4) at (-0.450000cm,-0.150000cm) {};
\node[fill=none,minimum size=0.1cm] (2-4-n) at (-0.450000cm,-0.150000cm) {};
\node[minimum size=0.04cm] (3-3) at (-0.450000cm,-0.750000cm) {};
\node[fill=none,minimum size=0.1cm] (3-3-n) at (-0.450000cm,-0.750000cm) {};
\node[minimum size=0.04cm] (3-4) at (-0.150000cm,-0.450000cm) {};
\node[fill=none,minimum size=0.1cm] (3-4-n) at (-0.150000cm,-0.450000cm) {};
\node[minimum size=0.04cm] (4-4) at (0.150000cm,-0.750000cm) {};
\node[fill=none,minimum size=0.1cm] (4-4-n) at (0.150000cm,-0.750000cm) {};
\draw[-] (3-4-n) edge (4-4-n);
\draw[-] (0-2-n) edge (0-3-n);
\draw[-] (4-4-n) edge (0-4-n);
\draw[-] (0-3-n) edge (1-3-n);
\draw[-] (0-4-n) edge ($(0-4-n)!0.5!(1-4-outside)$);
\draw[-] (1-4-n) edge ($(1-4-n)!0.5!(0-4-outside)$);
\draw[-] (1-3-n) edge ($(1-3-n)!0.5!(2-3-outside)$);
\draw[-] (2-3-n) edge ($(2-3-n)!0.5!(1-3-outside)$);
\draw[-] (2-3-n) edge (3-3-n);
\draw[-] (1-4-n) edge (2-4-n);
\draw[-] (3-3-n) edge (3-4-n);
\draw[-] (2-4-n) edge (0-2-n);
;
\end{scope}
\end{scope}
\begin{scope}[shift={(5.000000cm,-2.500000cm)}]
\begin{scope}[shift={(0.000000cm,-0.000000cm)}]
\tikzset{every node/.style={circle,fill,draw=black,minimum size = 0.1cm,inner sep=0cm}}
\node (0) at (0.000000,1.000000) {};
\node (1) at (-0.951057,0.309017) {};
\node (2) at (-0.587785,-0.809017) {};
\node (3) at (0.587785,-0.809017) {};
\node (4) at (0.951057,0.309017) {};
\path[-,every loop/.style={min distance=0.3cm}]
 (0) edge (2)
 (0) edge (4)
 (1) edge (2)
 (1) edge (3)
 (1) edge (4)
 (2) edge (3)
 (2) edge (4)
 (3) edge [in=346,out=266,loop] ()
 (3) edge (4)
 (4) edge [in=418,out=338,loop] ()
;\end{scope}
\begin{scope}[shift={(2.500000cm,-0.000000cm)}]
\tikzset{every node/.style={circle,fill,draw=black,inner sep=0cm}}
\node[minimum size=0.04cm] (0-0) at (-0.750000cm,0.750000cm) {};
\node[minimum size=0.04cm] (0-1) at (-0.450000cm,0.450000cm) {};
\node[minimum size=0.04cm] (0-2) at (-0.150000cm,0.150000cm) {};
\node[fill=none,minimum size=0.1cm] (0-2-n) at (-0.150000cm,0.150000cm) {};
\node[minimum size=0.04cm] (0-3) at (0.150000cm,-0.150000cm) {};
\node[minimum size=0.04cm] (0-4) at (0.450000cm,-0.450000cm) {};
\node[fill=none,minimum size=0.1cm] (0-4-n) at (0.450000cm,-0.450000cm) {};
\node[draw=none,fill=none] (0-4-outside) at (-1.050000cm,0.450000cm) {};
\node[minimum size=0.04cm] (1-1) at (-0.150000cm,0.750000cm) {};
\node[minimum size=0.04cm] (1-2) at (0.150000cm,0.450000cm) {};
\node[fill=none,minimum size=0.1cm] (1-2-n) at (0.150000cm,0.450000cm) {};
\node[minimum size=0.04cm] (1-3) at (0.450000cm,0.150000cm) {};
\node[fill=none,minimum size=0.1cm] (1-3-n) at (0.450000cm,0.150000cm) {};
\node[draw=none,fill=none] (1-3-outside) at (-1.050000cm,-0.150000cm) {};
\node[minimum size=0.04cm] (1-4) at (-0.750000cm,0.150000cm) {};
\node[fill=none,minimum size=0.1cm] (1-4-n) at (-0.750000cm,0.150000cm) {};
\node[draw=none,fill=none] (1-4-outside) at (0.750000cm,-0.150000cm) {};
\node[minimum size=0.04cm] (2-2) at (0.450000cm,0.750000cm) {};
\node[minimum size=0.04cm] (2-3) at (-0.750000cm,-0.450000cm) {};
\node[fill=none,minimum size=0.1cm] (2-3-n) at (-0.750000cm,-0.450000cm) {};
\node[draw=none,fill=none] (2-3-outside) at (0.750000cm,0.450000cm) {};
\node[minimum size=0.04cm] (2-4) at (-0.450000cm,-0.150000cm) {};
\node[fill=none,minimum size=0.1cm] (2-4-n) at (-0.450000cm,-0.150000cm) {};
\node[minimum size=0.04cm] (3-3) at (-0.450000cm,-0.750000cm) {};
\node[fill=none,minimum size=0.1cm] (3-3-n) at (-0.450000cm,-0.750000cm) {};
\node[minimum size=0.04cm] (3-4) at (-0.150000cm,-0.450000cm) {};
\node[fill=none,minimum size=0.1cm] (3-4-n) at (-0.150000cm,-0.450000cm) {};
\node[minimum size=0.04cm] (4-4) at (0.150000cm,-0.750000cm) {};
\node[fill=none,minimum size=0.1cm] (4-4-n) at (0.150000cm,-0.750000cm) {};
\draw[-] (3-4-n) edge (4-4-n);
\draw[-] (0-2-n) edge (1-2-n);
\draw[-] (0-4-n) edge ($(0-4-n)!0.5!(1-4-outside)$);
\draw[-] (1-4-n) edge ($(1-4-n)!0.5!(0-4-outside)$);
\draw[-] (1-3-n) edge ($(1-3-n)!0.5!(2-3-outside)$);
\draw[-] (2-3-n) edge ($(2-3-n)!0.5!(1-3-outside)$);
\draw[-] (4-4-n) edge (0-4-n);
\draw[-] (1-2-n) edge (1-3-n);
\draw[-] (2-3-n) edge (3-3-n);
\draw[-] (1-4-n) edge (2-4-n);
\draw[-] (3-3-n) edge (3-4-n);
\draw[-] (2-4-n) edge (0-2-n);
;
\end{scope}
\end{scope}
\begin{scope}[shift={(10.000000cm,-2.500000cm)}]
\begin{scope}[shift={(0.000000cm,-0.000000cm)}]
\tikzset{every node/.style={circle,fill,draw=black,minimum size = 0.1cm,inner sep=0cm}}
\node (0) at (0.000000,1.000000) {};
\node (1) at (-0.951057,0.309017) {};
\node (2) at (-0.587785,-0.809017) {};
\node (3) at (0.587785,-0.809017) {};
\node (4) at (0.951057,0.309017) {};
\path[-,every loop/.style={min distance=0.3cm}]
 (0) edge (1)
 (0) edge (2)
 (0) edge (4)
 (1) edge (2)
 (1) edge (3)
 (1) edge (4)
 (2) edge (3)
 (3) edge [in=346,out=266,loop] ()
 (3) edge (4)
 (4) edge [in=418,out=338,loop] ()
;\end{scope}
\begin{scope}[shift={(2.500000cm,-0.000000cm)}]
\tikzset{every node/.style={circle,fill,draw=black,inner sep=0cm}}
\node[minimum size=0.04cm] (0-0) at (-0.750000cm,0.750000cm) {};
\node[minimum size=0.04cm] (0-1) at (-0.450000cm,0.450000cm) {};
\node[fill=none,minimum size=0.1cm] (0-1-n) at (-0.450000cm,0.450000cm) {};
\node[minimum size=0.04cm] (0-2) at (-0.150000cm,0.150000cm) {};
\node[fill=none,minimum size=0.1cm] (0-2-n) at (-0.150000cm,0.150000cm) {};
\node[minimum size=0.04cm] (0-3) at (0.150000cm,-0.150000cm) {};
\node[minimum size=0.04cm] (0-4) at (0.450000cm,-0.450000cm) {};
\node[fill=none,minimum size=0.1cm] (0-4-n) at (0.450000cm,-0.450000cm) {};
\node[draw=none,fill=none] (0-4-outside) at (-1.050000cm,0.450000cm) {};
\node[minimum size=0.04cm] (1-1) at (-0.150000cm,0.750000cm) {};
\node[minimum size=0.04cm] (1-2) at (0.150000cm,0.450000cm) {};
\node[fill=none,minimum size=0.1cm] (1-2-n) at (0.150000cm,0.450000cm) {};
\node[minimum size=0.04cm] (1-3) at (0.450000cm,0.150000cm) {};
\node[fill=none,minimum size=0.1cm] (1-3-n) at (0.450000cm,0.150000cm) {};
\node[draw=none,fill=none] (1-3-outside) at (-1.050000cm,-0.150000cm) {};
\node[minimum size=0.04cm] (1-4) at (-0.750000cm,0.150000cm) {};
\node[fill=none,minimum size=0.1cm] (1-4-n) at (-0.750000cm,0.150000cm) {};
\node[draw=none,fill=none] (1-4-outside) at (0.750000cm,-0.150000cm) {};
\node[minimum size=0.04cm] (2-2) at (0.450000cm,0.750000cm) {};
\node[minimum size=0.04cm] (2-3) at (-0.750000cm,-0.450000cm) {};
\node[fill=none,minimum size=0.1cm] (2-3-n) at (-0.750000cm,-0.450000cm) {};
\node[draw=none,fill=none] (2-3-outside) at (0.750000cm,0.450000cm) {};
\node[minimum size=0.04cm] (2-4) at (-0.450000cm,-0.150000cm) {};
\node[minimum size=0.04cm] (3-3) at (-0.450000cm,-0.750000cm) {};
\node[fill=none,minimum size=0.1cm] (3-3-n) at (-0.450000cm,-0.750000cm) {};
\node[minimum size=0.04cm] (3-4) at (-0.150000cm,-0.450000cm) {};
\node[fill=none,minimum size=0.1cm] (3-4-n) at (-0.150000cm,-0.450000cm) {};
\node[minimum size=0.04cm] (4-4) at (0.150000cm,-0.750000cm) {};
\node[fill=none,minimum size=0.1cm] (4-4-n) at (0.150000cm,-0.750000cm) {};
\draw[-] (3-3-n) edge (3-4-n);
\draw[-] (0-1-n) edge (0-2-n);
\draw[-] (3-4-n) edge (4-4-n);
\draw[-] (0-2-n) edge (1-2-n);
\draw[-] (0-4-n) edge ($(0-4-n)!0.5!(1-4-outside)$);
\draw[-] (1-4-n) edge ($(1-4-n)!0.5!(0-4-outside)$);
\draw[-] (1-3-n) edge ($(1-3-n)!0.5!(2-3-outside)$);
\draw[-] (2-3-n) edge ($(2-3-n)!0.5!(1-3-outside)$);
\draw[-] (4-4-n) edge (0-4-n);
\draw[-] (1-2-n) edge (1-3-n);
\draw[-] (2-3-n) edge (3-3-n);
\draw[-] (1-4-n) edge (0-1-n);
;
\end{scope}
\end{scope}
\begin{scope}[shift={(0.000000cm,-5.000000cm)}]
\begin{scope}[shift={(0.000000cm,-0.000000cm)}]
\tikzset{every node/.style={circle,fill,draw=black,minimum size = 0.1cm,inner sep=0cm}}
\node (0) at (0.000000,1.000000) {};
\node (1) at (-0.951057,0.309017) {};
\node (2) at (-0.587785,-0.809017) {};
\node (3) at (0.587785,-0.809017) {};
\node (4) at (0.951057,0.309017) {};
\path[-,every loop/.style={min distance=0.3cm}]
 (0) edge (1)
 (0) edge (2)
 (0) edge (4)
 (1) edge (2)
 (1) edge (4)
 (2) edge [in=274,out=194,loop] ()
 (2) edge (3)
 (2) edge (4)
 (3) edge (4)
 (4) edge [in=418,out=338,loop] ()
;\end{scope}
\begin{scope}[shift={(2.500000cm,-0.000000cm)}]
\tikzset{every node/.style={circle,fill,draw=black,inner sep=0cm}}
\node[minimum size=0.04cm] (0-0) at (-0.750000cm,0.750000cm) {};
\node[minimum size=0.04cm] (0-1) at (-0.450000cm,0.450000cm) {};
\node[fill=none,minimum size=0.1cm] (0-1-n) at (-0.450000cm,0.450000cm) {};
\node[minimum size=0.04cm] (0-2) at (-0.150000cm,0.150000cm) {};
\node[fill=none,minimum size=0.1cm] (0-2-n) at (-0.150000cm,0.150000cm) {};
\node[minimum size=0.04cm] (0-3) at (0.150000cm,-0.150000cm) {};
\node[minimum size=0.04cm] (0-4) at (0.450000cm,-0.450000cm) {};
\node[fill=none,minimum size=0.1cm] (0-4-n) at (0.450000cm,-0.450000cm) {};
\node[draw=none,fill=none] (0-4-outside) at (-1.050000cm,0.450000cm) {};
\node[minimum size=0.04cm] (1-1) at (-0.150000cm,0.750000cm) {};
\node[minimum size=0.04cm] (1-2) at (0.150000cm,0.450000cm) {};
\node[fill=none,minimum size=0.1cm] (1-2-n) at (0.150000cm,0.450000cm) {};
\node[minimum size=0.04cm] (1-3) at (0.450000cm,0.150000cm) {};
\node[minimum size=0.04cm] (1-4) at (-0.750000cm,0.150000cm) {};
\node[fill=none,minimum size=0.1cm] (1-4-n) at (-0.750000cm,0.150000cm) {};
\node[draw=none,fill=none] (1-4-outside) at (0.750000cm,-0.150000cm) {};
\node[minimum size=0.04cm] (2-2) at (0.450000cm,0.750000cm) {};
\node[fill=none,minimum size=0.1cm] (2-2-n) at (0.450000cm,0.750000cm) {};
\node[draw=none,fill=none] (2-2-outside) at (-1.050000cm,-0.750000cm) {};
\node[minimum size=0.04cm] (2-3) at (-0.750000cm,-0.450000cm) {};
\node[fill=none,minimum size=0.1cm] (2-3-n) at (-0.750000cm,-0.450000cm) {};
\node[draw=none,fill=none] (2-3-outside) at (0.750000cm,0.450000cm) {};
\node[minimum size=0.04cm] (2-4) at (-0.450000cm,-0.150000cm) {};
\node[fill=none,minimum size=0.1cm] (2-4-n) at (-0.450000cm,-0.150000cm) {};
\node[minimum size=0.04cm] (3-3) at (-0.450000cm,-0.750000cm) {};
\node[minimum size=0.04cm] (3-4) at (-0.150000cm,-0.450000cm) {};
\node[fill=none,minimum size=0.1cm] (3-4-n) at (-0.150000cm,-0.450000cm) {};
\node[minimum size=0.04cm] (4-4) at (0.150000cm,-0.750000cm) {};
\node[fill=none,minimum size=0.1cm] (4-4-n) at (0.150000cm,-0.750000cm) {};
\draw[-] (2-4-n) edge (3-4-n);
\draw[-] (0-1-n) edge (0-2-n);
\draw[-] (3-4-n) edge (4-4-n);
\draw[-] (0-2-n) edge (1-2-n);
\draw[-] (0-4-n) edge ($(0-4-n)!0.5!(1-4-outside)$);
\draw[-] (1-4-n) edge ($(1-4-n)!0.5!(0-4-outside)$);
\draw[-] (2-2-n) edge ($(2-2-n)!0.5!(2-3-outside)$);
\draw[-] (2-3-n) edge ($(2-3-n)!0.5!(2-2-outside)$);
\draw[-] (4-4-n) edge (0-4-n);
\draw[-] (1-2-n) edge (2-2-n);
\draw[-] (2-3-n) edge (2-4-n);
\draw[-] (1-4-n) edge (0-1-n);
;
\end{scope}
\end{scope}
\begin{scope}[shift={(5.000000cm,-5.000000cm)}]
\begin{scope}[shift={(0.000000cm,-0.000000cm)}]
\tikzset{every node/.style={circle,fill,draw=black,minimum size = 0.1cm,inner sep=0cm}}
\node (0) at (0.000000,1.000000) {};
\node (1) at (-0.951057,0.309017) {};
\node (2) at (-0.587785,-0.809017) {};
\node (3) at (0.587785,-0.809017) {};
\node (4) at (0.951057,0.309017) {};
\path[-,every loop/.style={min distance=0.3cm}]
 (0) edge (2)
 (0) edge (4)
 (1) edge (2)
 (1) edge (4)
 (2) edge [in=274,out=194,loop] ()
 (2) edge (3)
 (2) edge (4)
 (3) edge [in=346,out=266,loop] ()
 (3) edge (4)
 (4) edge [in=418,out=338,loop] ()
;\end{scope}
\begin{scope}[shift={(2.500000cm,-0.000000cm)}]
\tikzset{every node/.style={circle,fill,draw=black,inner sep=0cm}}
\node[minimum size=0.04cm] (0-0) at (-0.750000cm,0.750000cm) {};
\node[minimum size=0.04cm] (0-1) at (-0.450000cm,0.450000cm) {};
\node[minimum size=0.04cm] (0-2) at (-0.150000cm,0.150000cm) {};
\node[fill=none,minimum size=0.1cm] (0-2-n) at (-0.150000cm,0.150000cm) {};
\node[minimum size=0.04cm] (0-3) at (0.150000cm,-0.150000cm) {};
\node[minimum size=0.04cm] (0-4) at (0.450000cm,-0.450000cm) {};
\node[fill=none,minimum size=0.1cm] (0-4-n) at (0.450000cm,-0.450000cm) {};
\node[draw=none,fill=none] (0-4-outside) at (-1.050000cm,0.450000cm) {};
\node[minimum size=0.04cm] (1-1) at (-0.150000cm,0.750000cm) {};
\node[minimum size=0.04cm] (1-2) at (0.150000cm,0.450000cm) {};
\node[fill=none,minimum size=0.1cm] (1-2-n) at (0.150000cm,0.450000cm) {};
\node[minimum size=0.04cm] (1-3) at (0.450000cm,0.150000cm) {};
\node[minimum size=0.04cm] (1-4) at (-0.750000cm,0.150000cm) {};
\node[fill=none,minimum size=0.1cm] (1-4-n) at (-0.750000cm,0.150000cm) {};
\node[draw=none,fill=none] (1-4-outside) at (0.750000cm,-0.150000cm) {};
\node[minimum size=0.04cm] (2-2) at (0.450000cm,0.750000cm) {};
\node[fill=none,minimum size=0.1cm] (2-2-n) at (0.450000cm,0.750000cm) {};
\node[draw=none,fill=none] (2-2-outside) at (-1.050000cm,-0.750000cm) {};
\node[minimum size=0.04cm] (2-3) at (-0.750000cm,-0.450000cm) {};
\node[fill=none,minimum size=0.1cm] (2-3-n) at (-0.750000cm,-0.450000cm) {};
\node[draw=none,fill=none] (2-3-outside) at (0.750000cm,0.450000cm) {};
\node[minimum size=0.04cm] (2-4) at (-0.450000cm,-0.150000cm) {};
\node[fill=none,minimum size=0.1cm] (2-4-n) at (-0.450000cm,-0.150000cm) {};
\node[minimum size=0.04cm] (3-3) at (-0.450000cm,-0.750000cm) {};
\node[fill=none,minimum size=0.1cm] (3-3-n) at (-0.450000cm,-0.750000cm) {};
\node[minimum size=0.04cm] (3-4) at (-0.150000cm,-0.450000cm) {};
\node[fill=none,minimum size=0.1cm] (3-4-n) at (-0.150000cm,-0.450000cm) {};
\node[minimum size=0.04cm] (4-4) at (0.150000cm,-0.750000cm) {};
\node[fill=none,minimum size=0.1cm] (4-4-n) at (0.150000cm,-0.750000cm) {};
\draw[-] (3-4-n) edge (4-4-n);
\draw[-] (0-2-n) edge (1-2-n);
\draw[-] (0-4-n) edge ($(0-4-n)!0.5!(1-4-outside)$);
\draw[-] (1-4-n) edge ($(1-4-n)!0.5!(0-4-outside)$);
\draw[-] (2-2-n) edge ($(2-2-n)!0.5!(2-3-outside)$);
\draw[-] (2-3-n) edge ($(2-3-n)!0.5!(2-2-outside)$);
\draw[-] (4-4-n) edge (0-4-n);
\draw[-] (1-2-n) edge (2-2-n);
\draw[-] (2-3-n) edge (3-3-n);
\draw[-] (1-4-n) edge (2-4-n);
\draw[-] (3-3-n) edge (3-4-n);
\draw[-] (2-4-n) edge (0-2-n);
;
\end{scope}
\end{scope}
\begin{scope}[shift={(10.000000cm,-5.000000cm)}]
\begin{scope}[shift={(0.000000cm,-0.000000cm)}]
\tikzset{every node/.style={circle,fill,draw=black,minimum size = 0.1cm,inner sep=0cm}}
\node (0) at (0.000000,1.000000) {};
\node (1) at (-0.951057,0.309017) {};
\node (2) at (-0.587785,-0.809017) {};
\node (3) at (0.587785,-0.809017) {};
\node (4) at (0.951057,0.309017) {};
\path[-,every loop/.style={min distance=0.3cm}]
 (0) edge (1)
 (0) edge (2)
 (0) edge (4)
 (1) edge (2)
 (1) edge (4)
 (2) edge [in=274,out=194,loop] ()
 (2) edge (3)
 (3) edge [in=346,out=266,loop] ()
 (3) edge (4)
 (4) edge [in=418,out=338,loop] ()
;\end{scope}
\begin{scope}[shift={(2.500000cm,-0.000000cm)}]
\tikzset{every node/.style={circle,fill,draw=black,inner sep=0cm}}
\node[minimum size=0.04cm] (0-0) at (-0.750000cm,0.750000cm) {};
\node[minimum size=0.04cm] (0-1) at (-0.450000cm,0.450000cm) {};
\node[fill=none,minimum size=0.1cm] (0-1-n) at (-0.450000cm,0.450000cm) {};
\node[minimum size=0.04cm] (0-2) at (-0.150000cm,0.150000cm) {};
\node[fill=none,minimum size=0.1cm] (0-2-n) at (-0.150000cm,0.150000cm) {};
\node[minimum size=0.04cm] (0-3) at (0.150000cm,-0.150000cm) {};
\node[minimum size=0.04cm] (0-4) at (0.450000cm,-0.450000cm) {};
\node[fill=none,minimum size=0.1cm] (0-4-n) at (0.450000cm,-0.450000cm) {};
\node[draw=none,fill=none] (0-4-outside) at (-1.050000cm,0.450000cm) {};
\node[minimum size=0.04cm] (1-1) at (-0.150000cm,0.750000cm) {};
\node[minimum size=0.04cm] (1-2) at (0.150000cm,0.450000cm) {};
\node[fill=none,minimum size=0.1cm] (1-2-n) at (0.150000cm,0.450000cm) {};
\node[minimum size=0.04cm] (1-3) at (0.450000cm,0.150000cm) {};
\node[minimum size=0.04cm] (1-4) at (-0.750000cm,0.150000cm) {};
\node[fill=none,minimum size=0.1cm] (1-4-n) at (-0.750000cm,0.150000cm) {};
\node[draw=none,fill=none] (1-4-outside) at (0.750000cm,-0.150000cm) {};
\node[minimum size=0.04cm] (2-2) at (0.450000cm,0.750000cm) {};
\node[fill=none,minimum size=0.1cm] (2-2-n) at (0.450000cm,0.750000cm) {};
\node[draw=none,fill=none] (2-2-outside) at (-1.050000cm,-0.750000cm) {};
\node[minimum size=0.04cm] (2-3) at (-0.750000cm,-0.450000cm) {};
\node[fill=none,minimum size=0.1cm] (2-3-n) at (-0.750000cm,-0.450000cm) {};
\node[draw=none,fill=none] (2-3-outside) at (0.750000cm,0.450000cm) {};
\node[minimum size=0.04cm] (2-4) at (-0.450000cm,-0.150000cm) {};
\node[minimum size=0.04cm] (3-3) at (-0.450000cm,-0.750000cm) {};
\node[fill=none,minimum size=0.1cm] (3-3-n) at (-0.450000cm,-0.750000cm) {};
\node[minimum size=0.04cm] (3-4) at (-0.150000cm,-0.450000cm) {};
\node[fill=none,minimum size=0.1cm] (3-4-n) at (-0.150000cm,-0.450000cm) {};
\node[minimum size=0.04cm] (4-4) at (0.150000cm,-0.750000cm) {};
\node[fill=none,minimum size=0.1cm] (4-4-n) at (0.150000cm,-0.750000cm) {};
\draw[-] (3-3-n) edge (3-4-n);
\draw[-] (0-1-n) edge (0-2-n);
\draw[-] (3-4-n) edge (4-4-n);
\draw[-] (0-2-n) edge (1-2-n);
\draw[-] (0-4-n) edge ($(0-4-n)!0.5!(1-4-outside)$);
\draw[-] (1-4-n) edge ($(1-4-n)!0.5!(0-4-outside)$);
\draw[-] (2-2-n) edge ($(2-2-n)!0.5!(2-3-outside)$);
\draw[-] (2-3-n) edge ($(2-3-n)!0.5!(2-2-outside)$);
\draw[-] (4-4-n) edge (0-4-n);
\draw[-] (1-2-n) edge (2-2-n);
\draw[-] (2-3-n) edge (3-3-n);
\draw[-] (1-4-n) edge (0-1-n);
;
\end{scope}
\end{scope}
\begin{scope}[shift={(0.000000cm,-7.500000cm)}]
\begin{scope}[shift={(0.000000cm,-0.000000cm)}]
\tikzset{every node/.style={circle,fill,draw=black,minimum size = 0.1cm,inner sep=0cm}}
\node (0) at (0.000000,1.000000) {};
\node (1) at (-0.951057,0.309017) {};
\node (2) at (-0.587785,-0.809017) {};
\node (3) at (0.587785,-0.809017) {};
\node (4) at (0.951057,0.309017) {};
\path[-,every loop/.style={min distance=0.3cm}]
 (0) edge (1)
 (0) edge (4)
 (1) edge [in=202,out=122,loop] ()
 (1) edge (2)
 (1) edge (3)
 (1) edge (4)
 (2) edge (3)
 (3) edge [in=346,out=266,loop] ()
 (3) edge (4)
 (4) edge [in=418,out=338,loop] ()
;\end{scope}
\begin{scope}[shift={(2.500000cm,-0.000000cm)}]
\tikzset{every node/.style={circle,fill,draw=black,inner sep=0cm}}
\node[minimum size=0.04cm] (0-0) at (-0.750000cm,0.750000cm) {};
\node[minimum size=0.04cm] (0-1) at (-0.450000cm,0.450000cm) {};
\node[fill=none,minimum size=0.1cm] (0-1-n) at (-0.450000cm,0.450000cm) {};
\node[minimum size=0.04cm] (0-2) at (-0.150000cm,0.150000cm) {};
\node[minimum size=0.04cm] (0-3) at (0.150000cm,-0.150000cm) {};
\node[minimum size=0.04cm] (0-4) at (0.450000cm,-0.450000cm) {};
\node[fill=none,minimum size=0.1cm] (0-4-n) at (0.450000cm,-0.450000cm) {};
\node[draw=none,fill=none] (0-4-outside) at (-1.050000cm,0.450000cm) {};
\node[minimum size=0.04cm] (1-1) at (-0.150000cm,0.750000cm) {};
\node[fill=none,minimum size=0.1cm] (1-1-n) at (-0.150000cm,0.750000cm) {};
\node[minimum size=0.04cm] (1-2) at (0.150000cm,0.450000cm) {};
\node[fill=none,minimum size=0.1cm] (1-2-n) at (0.150000cm,0.450000cm) {};
\node[minimum size=0.04cm] (1-3) at (0.450000cm,0.150000cm) {};
\node[fill=none,minimum size=0.1cm] (1-3-n) at (0.450000cm,0.150000cm) {};
\node[draw=none,fill=none] (1-3-outside) at (-1.050000cm,-0.150000cm) {};
\node[minimum size=0.04cm] (1-4) at (-0.750000cm,0.150000cm) {};
\node[fill=none,minimum size=0.1cm] (1-4-n) at (-0.750000cm,0.150000cm) {};
\node[draw=none,fill=none] (1-4-outside) at (0.750000cm,-0.150000cm) {};
\node[minimum size=0.04cm] (2-2) at (0.450000cm,0.750000cm) {};
\node[minimum size=0.04cm] (2-3) at (-0.750000cm,-0.450000cm) {};
\node[fill=none,minimum size=0.1cm] (2-3-n) at (-0.750000cm,-0.450000cm) {};
\node[draw=none,fill=none] (2-3-outside) at (0.750000cm,0.450000cm) {};
\node[minimum size=0.04cm] (2-4) at (-0.450000cm,-0.150000cm) {};
\node[minimum size=0.04cm] (3-3) at (-0.450000cm,-0.750000cm) {};
\node[fill=none,minimum size=0.1cm] (3-3-n) at (-0.450000cm,-0.750000cm) {};
\node[minimum size=0.04cm] (3-4) at (-0.150000cm,-0.450000cm) {};
\node[fill=none,minimum size=0.1cm] (3-4-n) at (-0.150000cm,-0.450000cm) {};
\node[minimum size=0.04cm] (4-4) at (0.150000cm,-0.750000cm) {};
\node[fill=none,minimum size=0.1cm] (4-4-n) at (0.150000cm,-0.750000cm) {};
\draw[-] (3-3-n) edge (3-4-n);
\draw[-] (0-1-n) edge (1-1-n);
\draw[-] (0-4-n) edge ($(0-4-n)!0.5!(1-4-outside)$);
\draw[-] (1-4-n) edge ($(1-4-n)!0.5!(0-4-outside)$);
\draw[-] (1-3-n) edge ($(1-3-n)!0.5!(2-3-outside)$);
\draw[-] (2-3-n) edge ($(2-3-n)!0.5!(1-3-outside)$);
\draw[-] (3-4-n) edge (4-4-n);
\draw[-] (1-1-n) edge (1-2-n);
\draw[-] (4-4-n) edge (0-4-n);
\draw[-] (1-2-n) edge (1-3-n);
\draw[-] (2-3-n) edge (3-3-n);
\draw[-] (1-4-n) edge (0-1-n);
;
\end{scope}
\end{scope}
\begin{scope}[shift={(5.000000cm,-7.500000cm)}]
\begin{scope}[shift={(0.000000cm,-0.000000cm)}]
\tikzset{every node/.style={circle,fill,draw=black,minimum size = 0.1cm,inner sep=0cm}}
\node (0) at (0.000000,1.000000) {};
\node (1) at (-0.951057,0.309017) {};
\node (2) at (-0.587785,-0.809017) {};
\node (3) at (0.587785,-0.809017) {};
\node (4) at (0.951057,0.309017) {};
\path[-,every loop/.style={min distance=0.3cm}]
 (0) edge (1)
 (0) edge (4)
 (1) edge [in=202,out=122,loop] ()
 (1) edge (2)
 (1) edge (4)
 (2) edge [in=274,out=194,loop] ()
 (2) edge (3)
 (3) edge [in=346,out=266,loop] ()
 (3) edge (4)
 (4) edge [in=418,out=338,loop] ()
;\end{scope}
\begin{scope}[shift={(2.500000cm,-0.000000cm)}]
\tikzset{every node/.style={circle,fill,draw=black,inner sep=0cm}}
\node[minimum size=0.04cm] (0-0) at (-0.750000cm,0.750000cm) {};
\node[minimum size=0.04cm] (0-1) at (-0.450000cm,0.450000cm) {};
\node[fill=none,minimum size=0.1cm] (0-1-n) at (-0.450000cm,0.450000cm) {};
\node[minimum size=0.04cm] (0-2) at (-0.150000cm,0.150000cm) {};
\node[minimum size=0.04cm] (0-3) at (0.150000cm,-0.150000cm) {};
\node[minimum size=0.04cm] (0-4) at (0.450000cm,-0.450000cm) {};
\node[fill=none,minimum size=0.1cm] (0-4-n) at (0.450000cm,-0.450000cm) {};
\node[draw=none,fill=none] (0-4-outside) at (-1.050000cm,0.450000cm) {};
\node[minimum size=0.04cm] (1-1) at (-0.150000cm,0.750000cm) {};
\node[fill=none,minimum size=0.1cm] (1-1-n) at (-0.150000cm,0.750000cm) {};
\node[minimum size=0.04cm] (1-2) at (0.150000cm,0.450000cm) {};
\node[fill=none,minimum size=0.1cm] (1-2-n) at (0.150000cm,0.450000cm) {};
\node[minimum size=0.04cm] (1-3) at (0.450000cm,0.150000cm) {};
\node[minimum size=0.04cm] (1-4) at (-0.750000cm,0.150000cm) {};
\node[fill=none,minimum size=0.1cm] (1-4-n) at (-0.750000cm,0.150000cm) {};
\node[draw=none,fill=none] (1-4-outside) at (0.750000cm,-0.150000cm) {};
\node[minimum size=0.04cm] (2-2) at (0.450000cm,0.750000cm) {};
\node[fill=none,minimum size=0.1cm] (2-2-n) at (0.450000cm,0.750000cm) {};
\node[draw=none,fill=none] (2-2-outside) at (-1.050000cm,-0.750000cm) {};
\node[minimum size=0.04cm] (2-3) at (-0.750000cm,-0.450000cm) {};
\node[fill=none,minimum size=0.1cm] (2-3-n) at (-0.750000cm,-0.450000cm) {};
\node[draw=none,fill=none] (2-3-outside) at (0.750000cm,0.450000cm) {};
\node[minimum size=0.04cm] (2-4) at (-0.450000cm,-0.150000cm) {};
\node[minimum size=0.04cm] (3-3) at (-0.450000cm,-0.750000cm) {};
\node[fill=none,minimum size=0.1cm] (3-3-n) at (-0.450000cm,-0.750000cm) {};
\node[minimum size=0.04cm] (3-4) at (-0.150000cm,-0.450000cm) {};
\node[fill=none,minimum size=0.1cm] (3-4-n) at (-0.150000cm,-0.450000cm) {};
\node[minimum size=0.04cm] (4-4) at (0.150000cm,-0.750000cm) {};
\node[fill=none,minimum size=0.1cm] (4-4-n) at (0.150000cm,-0.750000cm) {};
\draw[-] (3-3-n) edge (3-4-n);
\draw[-] (0-1-n) edge (1-1-n);
\draw[-] (0-4-n) edge ($(0-4-n)!0.5!(1-4-outside)$);
\draw[-] (1-4-n) edge ($(1-4-n)!0.5!(0-4-outside)$);
\draw[-] (2-2-n) edge ($(2-2-n)!0.5!(2-3-outside)$);
\draw[-] (2-3-n) edge ($(2-3-n)!0.5!(2-2-outside)$);
\draw[-] (3-4-n) edge (4-4-n);
\draw[-] (1-1-n) edge (1-2-n);
\draw[-] (4-4-n) edge (0-4-n);
\draw[-] (1-2-n) edge (2-2-n);
\draw[-] (2-3-n) edge (3-3-n);
\draw[-] (1-4-n) edge (0-1-n);
;
\end{scope}
\end{scope}
\begin{scope}[shift={(10.000000cm,-7.500000cm)}]
\begin{scope}[shift={(0.000000cm,-0.000000cm)}]
\tikzset{every node/.style={circle,fill,draw=black,minimum size = 0.1cm,inner sep=0cm}}
\node (0) at (0.000000,1.000000) {};
\node (1) at (-0.951057,0.309017) {};
\node (2) at (-0.587785,-0.809017) {};
\node (3) at (0.587785,-0.809017) {};
\node (4) at (0.951057,0.309017) {};
\path[-,every loop/.style={min distance=0.3cm}]
 (0) edge [in=130,out=50,loop] ()
 (0) edge (1)
 (0) edge (4)
 (1) edge [in=202,out=122,loop] ()
 (1) edge (2)
 (2) edge [in=274,out=194,loop] ()
 (2) edge (3)
 (3) edge [in=346,out=266,loop] ()
 (3) edge (4)
 (4) edge [in=418,out=338,loop] ()
;\end{scope}
\begin{scope}[shift={(2.500000cm,-0.000000cm)}]
\tikzset{every node/.style={circle,fill,draw=black,inner sep=0cm}}
\node[minimum size=0.04cm] (0-0) at (-0.750000cm,0.750000cm) {};
\node[fill=none,minimum size=0.1cm] (0-0-n) at (-0.750000cm,0.750000cm) {};
\node[draw=none,fill=none] (0-0-outside) at (0.750000cm,-0.750000cm) {};
\node[minimum size=0.04cm] (0-1) at (-0.450000cm,0.450000cm) {};
\node[fill=none,minimum size=0.1cm] (0-1-n) at (-0.450000cm,0.450000cm) {};
\node[minimum size=0.04cm] (0-2) at (-0.150000cm,0.150000cm) {};
\node[minimum size=0.04cm] (0-3) at (0.150000cm,-0.150000cm) {};
\node[minimum size=0.04cm] (0-4) at (0.450000cm,-0.450000cm) {};
\node[fill=none,minimum size=0.1cm] (0-4-n) at (0.450000cm,-0.450000cm) {};
\node[draw=none,fill=none] (0-4-outside) at (-1.050000cm,0.450000cm) {};
\node[minimum size=0.04cm] (1-1) at (-0.150000cm,0.750000cm) {};
\node[fill=none,minimum size=0.1cm] (1-1-n) at (-0.150000cm,0.750000cm) {};
\node[minimum size=0.04cm] (1-2) at (0.150000cm,0.450000cm) {};
\node[fill=none,minimum size=0.1cm] (1-2-n) at (0.150000cm,0.450000cm) {};
\node[minimum size=0.04cm] (1-3) at (0.450000cm,0.150000cm) {};
\node[minimum size=0.04cm] (1-4) at (-0.750000cm,0.150000cm) {};
\node[minimum size=0.04cm] (2-2) at (0.450000cm,0.750000cm) {};
\node[fill=none,minimum size=0.1cm] (2-2-n) at (0.450000cm,0.750000cm) {};
\node[draw=none,fill=none] (2-2-outside) at (-1.050000cm,-0.750000cm) {};
\node[minimum size=0.04cm] (2-3) at (-0.750000cm,-0.450000cm) {};
\node[fill=none,minimum size=0.1cm] (2-3-n) at (-0.750000cm,-0.450000cm) {};
\node[draw=none,fill=none] (2-3-outside) at (0.750000cm,0.450000cm) {};
\node[minimum size=0.04cm] (2-4) at (-0.450000cm,-0.150000cm) {};
\node[minimum size=0.04cm] (3-3) at (-0.450000cm,-0.750000cm) {};
\node[fill=none,minimum size=0.1cm] (3-3-n) at (-0.450000cm,-0.750000cm) {};
\node[minimum size=0.04cm] (3-4) at (-0.150000cm,-0.450000cm) {};
\node[fill=none,minimum size=0.1cm] (3-4-n) at (-0.150000cm,-0.450000cm) {};
\node[minimum size=0.04cm] (4-4) at (0.150000cm,-0.750000cm) {};
\node[fill=none,minimum size=0.1cm] (4-4-n) at (0.150000cm,-0.750000cm) {};
\draw[-] (2-3-n) edge (3-3-n);
\draw[-] (0-0-n) edge (0-1-n);
\draw[-] (3-3-n) edge (3-4-n);
\draw[-] (0-1-n) edge (1-1-n);
\draw[-] (0-4-n) edge ($(0-4-n)!0.5!(0-0-outside)$);
\draw[-] (0-0-n) edge ($(0-0-n)!0.5!(0-4-outside)$);
\draw[-] (2-2-n) edge ($(2-2-n)!0.5!(2-3-outside)$);
\draw[-] (2-3-n) edge ($(2-3-n)!0.5!(2-2-outside)$);
\draw[-] (3-4-n) edge (4-4-n);
\draw[-] (1-1-n) edge (1-2-n);
\draw[-] (4-4-n) edge (0-4-n);
\draw[-] (1-2-n) edge (2-2-n);
;
\end{scope}
\end{scope}

%% file: regions.tex
\tikzset{every node/.style={circle,fill,draw=black,inner sep=0cm}}
\fill[fill=gray!20] (3cm,-11cm) -- (0cm,-14cm) -- (0cm,-19cm) -- (11cm,-19cm);
\fill[fill=gray!20] (18cm,-4cm) -- (18cm,-0cm) -- (14cm,-0cm);
\fill[fill=gray!20,pattern=horizontal lines gray] (3cm,-7cm) -- (0cm,-4cm) -- (0cm,-0cm) -- (10cm,-0cm);
\fill[fill=gray!20,pattern=horizontal lines gray] (18cm,-16cm) -- (18cm,-19cm) -- (15cm,-19cm);\node[minimum size=0.04cm] (-9-9) at (0cm,-18cm) {};
\node[minimum size=0.04cm] (-9-10) at (1cm,-19cm) {};
\node[minimum size=0.04cm] (-8-8) at (0cm,-16cm) {};
\node[minimum size=0.04cm] (-8-9) at (1cm,-17cm) {};
\node[minimum size=0.04cm] (-8-10) at (2cm,-18cm) {};
\node[minimum size=0.04cm] (-8-11) at (3cm,-19cm) {};
\node[minimum size=0.04cm] (-7-7) at (0cm,-14cm) {};
\node[minimum size=0.04cm] (-7-8) at (1cm,-15cm) {};
\node[minimum size=0.04cm] (-7-9) at (2cm,-16cm) {};
\node[minimum size=0.04cm] (-7-10) at (3cm,-17cm) {};
\node[minimum size=0.04cm] (-7-11) at (4cm,-18cm) {};
\node[minimum size=0.04cm] (-7-12) at (5cm,-19cm) {};
\node[minimum size=0.04cm] (-6-6) at (0cm,-12cm) {};
\node[minimum size=0.04cm] (-6-7) at (1cm,-13cm) {};
\node[minimum size=0.04cm] (-6-8) at (2cm,-14cm) {};
\node[minimum size=0.04cm] (-6-9) at (3cm,-15cm) {};
\node[minimum size=0.04cm] (-6-10) at (4cm,-16cm) {};
\node[minimum size=0.04cm] (-6-11) at (5cm,-17cm) {};
\node[minimum size=0.04cm] (-6-12) at (6cm,-18cm) {};
\node[minimum size=0.04cm] (-6-13) at (7cm,-19cm) {};
\node[minimum size=0.04cm] (-5-5) at (0cm,-10cm) {};
\node[minimum size=0.04cm] (-5-6) at (1cm,-11cm) {};
\node[minimum size=0.04cm] (-5-7) at (2cm,-12cm) {};
\node[minimum size=0.04cm] (-5-8) at (3cm,-13cm) {};
\node[minimum size=0.04cm] (-5-9) at (4cm,-14cm) {};
\node[minimum size=0.04cm] (-5-10) at (5cm,-15cm) {};
\node[minimum size=0.04cm] (-5-11) at (6cm,-16cm) {};
\node[minimum size=0.04cm] (-5-12) at (7cm,-17cm) {};
\node[minimum size=0.04cm] (-5-13) at (8cm,-18cm) {};
\node[minimum size=0.04cm] (-5-14) at (9cm,-19cm) {};
\node[minimum size=0.04cm] (-4-4) at (0cm,-8cm) {};
\node[minimum size=0.04cm] (-4-5) at (1cm,-9cm) {};
\node[minimum size=0.04cm] (-4-6) at (2cm,-10cm) {};
\node[minimum size=0.04cm] (-4-7) at (3cm,-11cm) {};
\node[minimum size=0.04cm] (-4-8) at (4cm,-12cm) {};
\node[minimum size=0.04cm] (-4-9) at (5cm,-13cm) {};
\node[minimum size=0.04cm] (-4-10) at (6cm,-14cm) {};
\node[minimum size=0.04cm] (-4-11) at (7cm,-15cm) {};
\node[minimum size=0.04cm] (-4-12) at (8cm,-16cm) {};
\node[minimum size=0.04cm] (-4-13) at (9cm,-17cm) {};
\node[minimum size=0.04cm] (-4-14) at (10cm,-18cm) {};
\node[minimum size=0.04cm] (-4-15) at (11cm,-19cm) {};
\node[minimum size=0.04cm] (-3-3) at (0cm,-6cm) {};
\node[minimum size=0.04cm] (-3-4) at (1cm,-7cm) {};
\node[minimum size=0.04cm] (-3-5) at (2cm,-8cm) {};
\node[minimum size=0.04cm] (-3-6) at (3cm,-9cm) {};
\node[minimum size=0.04cm] (-3-7) at (4cm,-10cm) {};
\node[minimum size=0.04cm] (-3-8) at (5cm,-11cm) {};
\node[minimum size=0.04cm] (-3-9) at (6cm,-12cm) {};
\node[minimum size=0.04cm] (-3-10) at (7cm,-13cm) {};
\node[minimum size=0.04cm] (-3-11) at (8cm,-14cm) {};
\node[minimum size=0.04cm] (-3-12) at (9cm,-15cm) {};
\node[minimum size=0.04cm] (-3-13) at (10cm,-16cm) {};
\node[minimum size=0.04cm] (-3-14) at (11cm,-17cm) {};
\node[minimum size=0.04cm] (-3-15) at (12cm,-18cm) {};
\node[minimum size=0.04cm] (-3-16) at (13cm,-19cm) {};
\node[minimum size=0.04cm] (-2-2) at (0cm,-4cm) {};
\node[minimum size=0.04cm] (-2-3) at (1cm,-5cm) {};
\node[minimum size=0.04cm] (-2-4) at (2cm,-6cm) {};
\node[minimum size=0.04cm] (-2-5) at (3cm,-7cm) {};
\node[minimum size=0.04cm] (-2-6) at (4cm,-8cm) {};
\node[minimum size=0.04cm] (-2-7) at (5cm,-9cm) {};
\node[minimum size=0.04cm] (-2-8) at (6cm,-10cm) {};
\node[minimum size=0.04cm] (-2-9) at (7cm,-11cm) {};
\node[minimum size=0.04cm] (-2-10) at (8cm,-12cm) {};
\node[minimum size=0.04cm] (-2-11) at (9cm,-13cm) {};
\node[minimum size=0.04cm] (-2-12) at (10cm,-14cm) {};
\node[minimum size=0.04cm] (-2-13) at (11cm,-15cm) {};
\node[minimum size=0.04cm] (-2-14) at (12cm,-16cm) {};
\node[minimum size=0.04cm] (-2-15) at (13cm,-17cm) {};
\node[minimum size=0.04cm] (-2-16) at (14cm,-18cm) {};
\node[minimum size=0.04cm] (-2-17) at (15cm,-19cm) {};
\node[minimum size=0.04cm] (-1-1) at (0cm,-2cm) {};
\node[minimum size=0.04cm] (-1-2) at (1cm,-3cm) {};
\node[minimum size=0.04cm] (-1-3) at (2cm,-4cm) {};
\node[minimum size=0.04cm] (-1-4) at (3cm,-5cm) {};
\node[minimum size=0.04cm] (-1-5) at (4cm,-6cm) {};
\node[minimum size=0.04cm] (-1-6) at (5cm,-7cm) {};
\node[minimum size=0.04cm] (-1-7) at (6cm,-8cm) {};
\node[minimum size=0.04cm] (-1-8) at (7cm,-9cm) {};
\node[minimum size=0.04cm] (-1-9) at (8cm,-10cm) {};
\node[minimum size=0.04cm] (-1-10) at (9cm,-11cm) {};
\node[minimum size=0.04cm] (-1-11) at (10cm,-12cm) {};
\node[minimum size=0.04cm] (-1-12) at (11cm,-13cm) {};
\node[minimum size=0.04cm] (-1-13) at (12cm,-14cm) {};
\node[minimum size=0.04cm] (-1-14) at (13cm,-15cm) {};
\node[minimum size=0.04cm] (-1-15) at (14cm,-16cm) {};
\node[minimum size=0.04cm] (-1-16) at (15cm,-17cm) {};
\node[minimum size=0.04cm] (-1-17) at (16cm,-18cm) {};
\node[minimum size=0.04cm] (-1-18) at (17cm,-19cm) {};
\node[minimum size=0.04cm] (0-0) at (0cm,0cm) {};
\node[minimum size=0.04cm] (0-1) at (1cm,-1cm) {};
\node[minimum size=0.04cm] (0-2) at (2cm,-2cm) {};
\node[minimum size=0.04cm] (0-3) at (3cm,-3cm) {};
\node[minimum size=0.04cm] (0-4) at (4cm,-4cm) {};
\node[minimum size=0.04cm] (0-5) at (5cm,-5cm) {};
\node[minimum size=0.04cm] (0-6) at (6cm,-6cm) {};
\node[minimum size=0.04cm] (0-7) at (7cm,-7cm) {};
\node[minimum size=0.04cm] (0-8) at (8cm,-8cm) {};
\node[minimum size=0.04cm] (0-9) at (9cm,-9cm) {};
\node[minimum size=0.04cm] (0-10) at (10cm,-10cm) {};
\node[minimum size=0.04cm] (0-11) at (11cm,-11cm) {};
\node[minimum size=0.04cm] (0-12) at (12cm,-12cm) {};
\node[minimum size=0.04cm] (0-13) at (13cm,-13cm) {};
\node[minimum size=0.04cm] (0-14) at (14cm,-14cm) {};
\node[minimum size=0.04cm] (0-15) at (15cm,-15cm) {};
\node[minimum size=0.04cm] (0-16) at (16cm,-16cm) {};
\node[minimum size=0.04cm] (0-17) at (17cm,-17cm) {};
\node[minimum size=0.04cm] (0-18) at (18cm,-18cm) {};
\node[minimum size=0.04cm] (1-1) at (2cm,0cm) {};
\node[minimum size=0.04cm] (1-2) at (3cm,-1cm) {};
\node[minimum size=0.04cm] (1-3) at (4cm,-2cm) {};
\node[minimum size=0.04cm] (1-4) at (5cm,-3cm) {};
\node[minimum size=0.04cm] (1-5) at (6cm,-4cm) {};
\node[minimum size=0.04cm] (1-6) at (7cm,-5cm) {};
\node[minimum size=0.04cm] (1-7) at (8cm,-6cm) {};
\node[minimum size=0.04cm] (1-8) at (9cm,-7cm) {};
\node[minimum size=0.04cm] (1-9) at (10cm,-8cm) {};
\node[minimum size=0.04cm] (1-10) at (11cm,-9cm) {};
\node[minimum size=0.04cm] (1-11) at (12cm,-10cm) {};
\node[minimum size=0.04cm] (1-12) at (13cm,-11cm) {};
\node[minimum size=0.04cm] (1-13) at (14cm,-12cm) {};
\node[minimum size=0.04cm] (1-14) at (15cm,-13cm) {};
\node[minimum size=0.04cm] (1-15) at (16cm,-14cm) {};
\node[minimum size=0.04cm] (1-16) at (17cm,-15cm) {};
\node[minimum size=0.04cm] (1-17) at (18cm,-16cm) {};
\node[minimum size=0.04cm] (2-2) at (4cm,0cm) {};
\node[minimum size=0.04cm] (2-3) at (5cm,-1cm) {};
\node[minimum size=0.04cm] (2-4) at (6cm,-2cm) {};
\node[minimum size=0.04cm] (2-5) at (7cm,-3cm) {};
\node[minimum size=0.04cm] (2-6) at (8cm,-4cm) {};
\node[minimum size=0.04cm] (2-7) at (9cm,-5cm) {};
\node[minimum size=0.04cm] (2-8) at (10cm,-6cm) {};
\node[minimum size=0.04cm] (2-9) at (11cm,-7cm) {};
\node[minimum size=0.04cm] (2-10) at (12cm,-8cm) {};
\node[minimum size=0.04cm] (2-11) at (13cm,-9cm) {};
\node[minimum size=0.04cm] (2-12) at (14cm,-10cm) {};
\node[minimum size=0.04cm] (2-13) at (15cm,-11cm) {};
\node[minimum size=0.04cm] (2-14) at (16cm,-12cm) {};
\node[minimum size=0.04cm] (2-15) at (17cm,-13cm) {};
\node[minimum size=0.04cm] (2-16) at (18cm,-14cm) {};
\node[minimum size=0.04cm] (3-3) at (6cm,0cm) {};
\node[minimum size=0.04cm] (3-4) at (7cm,-1cm) {};
\node[minimum size=0.04cm] (3-5) at (8cm,-2cm) {};
\node[minimum size=0.04cm] (3-6) at (9cm,-3cm) {};
\node[minimum size=0.04cm] (3-7) at (10cm,-4cm) {};
\node[minimum size=0.04cm] (3-8) at (11cm,-5cm) {};
\node[minimum size=0.04cm] (3-9) at (12cm,-6cm) {};
\node[minimum size=0.04cm] (3-10) at (13cm,-7cm) {};
\node[minimum size=0.04cm] (3-11) at (14cm,-8cm) {};
\node[minimum size=0.04cm] (3-12) at (15cm,-9cm) {};
\node[minimum size=0.04cm] (3-13) at (16cm,-10cm) {};
\node[minimum size=0.04cm] (3-14) at (17cm,-11cm) {};
\node[minimum size=0.04cm] (3-15) at (18cm,-12cm) {};
\node[minimum size=0.04cm] (4-4) at (8cm,0cm) {};
\node[minimum size=0.04cm] (4-5) at (9cm,-1cm) {};
\node[minimum size=0.04cm] (4-6) at (10cm,-2cm) {};
\node[minimum size=0.04cm] (4-7) at (11cm,-3cm) {};
\node[minimum size=0.04cm] (4-8) at (12cm,-4cm) {};
\node[minimum size=0.04cm] (4-9) at (13cm,-5cm) {};
\node[minimum size=0.04cm] (4-10) at (14cm,-6cm) {};
\node[minimum size=0.04cm] (4-11) at (15cm,-7cm) {};
\node[minimum size=0.04cm] (4-12) at (16cm,-8cm) {};
\node[minimum size=0.04cm] (4-13) at (17cm,-9cm) {};
\node[minimum size=0.04cm] (4-14) at (18cm,-10cm) {};
\node[minimum size=0.04cm] (5-5) at (10cm,0cm) {};
\node[minimum size=0.04cm] (5-6) at (11cm,-1cm) {};
\node[minimum size=0.04cm] (5-7) at (12cm,-2cm) {};
\node[minimum size=0.04cm] (5-8) at (13cm,-3cm) {};
\node[minimum size=0.04cm] (5-9) at (14cm,-4cm) {};
\node[minimum size=0.04cm] (5-10) at (15cm,-5cm) {};
\node[minimum size=0.04cm] (5-11) at (16cm,-6cm) {};
\node[minimum size=0.04cm] (5-12) at (17cm,-7cm) {};
\node[minimum size=0.04cm] (5-13) at (18cm,-8cm) {};
\node[minimum size=0.04cm] (6-6) at (12cm,0cm) {};
\node[minimum size=0.04cm] (6-7) at (13cm,-1cm) {};
\node[minimum size=0.04cm] (6-8) at (14cm,-2cm) {};
\node[minimum size=0.04cm] (6-9) at (15cm,-3cm) {};
\node[minimum size=0.04cm] (6-10) at (16cm,-4cm) {};
\node[minimum size=0.04cm] (6-11) at (17cm,-5cm) {};
\node[minimum size=0.04cm] (6-12) at (18cm,-6cm) {};
\node[minimum size=0.04cm] (7-7) at (14cm,0cm) {};
\node[minimum size=0.04cm] (7-8) at (15cm,-1cm) {};
\node[minimum size=0.04cm] (7-9) at (16cm,-2cm) {};
\node[minimum size=0.04cm] (7-10) at (17cm,-3cm) {};
\node[minimum size=0.04cm] (7-11) at (18cm,-4cm) {};
\node[minimum size=0.04cm] (8-8) at (16cm,0cm) {};
\node[minimum size=0.04cm] (8-9) at (17cm,-1cm) {};
\node[minimum size=0.04cm] (8-10) at (18cm,-2cm) {};
\node[minimum size=0.04cm] (9-9) at (18cm,0cm) {};
\node[fill=none,minimum size=0.1cm,label=left:$\alpha$] (-3-6-n) at (3cm,-9cm) {};
\node[draw=none,fill=none] (I) at (3cm,-4cm) {$R_1^\alpha$};
\node[draw=none,fill=none] (II) at (3cm,-16cm) {$R_2^\alpha$};
\node[draw=none,fill=none] (III) at (12cm,-9cm) {$R_3^\alpha$};

%% file: plot6one.tex
\begin{scope}[shift={(0.000000cm,-0.000000cm)}]
\tikzset{every node/.style={circle,fill,draw=black,minimum size = 0.1cm,inner sep=0cm}}
\node (0) at (0.000000,1.000000) {};
\node (1) at (-0.866025,0.500000) {};
\node (2) at (-0.866025,-0.500000) {};
\node (3) at (0.000000,-1.000000) {};
\node (4) at (0.866025,-0.500000) {};
\node (5) at (0.866025,0.500000) {};
\path[-,every loop/.style={min distance=0.3cm}]
 (0) edge [in=130,out=50,loop] ()
 (0) edge (1)
 (0) edge (4)
 (0) edge (5)
 (1) edge [in=190,out=110,loop] ()
 (1) edge (2)
 (1) edge (3)
 (1) edge (4)
 (2) edge (4)
 (3) edge (4)
 (3) edge (5)
 (4) edge (5)
;\end{scope}
\begin{scope}[shift={(2.800000cm,-0.000000cm)}]
\tikzset{every node/.style={circle,fill,draw=black,inner sep=0cm}}
\node[minimum size=0.04cm,label=left:0] (0-0) at (-0.900000cm,0.900000cm) {};
\node[minimum size=0.04cm,label=right:0] (0-0-outside) at (0.900000cm,-0.900000cm) {};
\node[minimum size=0.04cm] (0-1) at (-0.600000cm,0.600000cm) {};
\node[minimum size=0.04cm] (0-2) at (-0.300000cm,0.300000cm) {};
\node[minimum size=0.04cm] (0-3) at (0.000000cm,0.000000cm) {};
\node[minimum size=0.04cm] (0-4) at (0.300000cm,-0.300000cm) {};
\node[minimum size=0.04cm] (0-5) at (0.600000cm,-0.600000cm) {};
\node[minimum size=0.04cm] (1-1) at (-0.300000cm,0.900000cm) {};
\node[minimum size=0.04cm] (1-2) at (0.000000cm,0.600000cm) {};
\node[minimum size=0.04cm] (1-3) at (0.300000cm,0.300000cm) {};
\node[minimum size=0.04cm] (1-4) at (0.600000cm,0.000000cm) {};
\node[minimum size=0.04cm,label=left:1] (1-5) at (-0.900000cm,0.300000cm) {};
\node[minimum size=0.04cm,label=right:1] (1-5-outside) at (0.900000cm,-0.300000cm) {};
\node[minimum size=0.04cm] (2-2) at (0.300000cm,0.900000cm) {};
\node[minimum size=0.04cm] (2-3) at (0.600000cm,0.600000cm) {};
\node[minimum size=0.04cm,label=left:2] (2-4) at (-0.900000cm,-0.300000cm) {};
\node[minimum size=0.04cm,label=right:2] (2-4-outside) at (0.900000cm,0.300000cm) {};
\node[minimum size=0.04cm] (2-5) at (-0.600000cm,0.000000cm) {};
\node[minimum size=0.04cm,label=left:3] (3-3) at (-0.900000cm,-0.900000cm) {};
\node[minimum size=0.04cm,label=right:3] (3-3-outside) at (0.900000cm,0.900000cm) {};
\node[minimum size=0.04cm] (3-4) at (-0.600000cm,-0.600000cm) {};
\node[minimum size=0.04cm] (3-5) at (-0.300000cm,-0.300000cm) {};
\node[minimum size=0.04cm] (4-4) at (-0.300000cm,-0.900000cm) {};
\node[minimum size=0.04cm] (4-5) at (0.000000cm,-0.600000cm) {};
\node[minimum size=0.04cm] (5-5) at (0.300000cm,-0.900000cm) {};
\draw[-] (2-4) edge (3-4);
\draw[-] (0-0) edge (0-1);
\draw[-] (3-4) edge (3-5);
\draw[-] (0-1) edge (1-1);
\draw[-] (0-4) edge (0-5);
\draw[-] (1-3) edge (1-4);
\draw[-] (0-5) edge (0-0-outside);
\draw[-] (1-4) edge (2-4-outside);
\draw[-] (3-5) edge (4-5);
\draw[-] (1-1) edge (1-2);
\draw[-] (4-5) edge (0-4);
\draw[-] (1-2) edge (1-3);
;
\end{scope}

%% file: plot7one.tex
\begin{scope}[shift={(0.000000cm,-0.000000cm)}]
\tikzset{every node/.style={circle,fill,draw=black,minimum size = 0.1cm,inner sep=0cm}}
\node (0) at (0.000000,1.000000) {};
\node (1) at (-0.781831,0.623490) {};
\node (2) at (-0.974928,-0.222521) {};
\node (3) at (-0.433884,-0.900969) {};
\node (4) at (0.433884,-0.900969) {};
\node (5) at (0.974928,-0.222521) {};
\node (6) at (0.781831,0.623490) {};
\path[-,every loop/.style={min distance=0.3cm}]
 (0) edge (2)
 (0) edge (3)
 (0) edge (4)
 (0) edge (5)
 (0) edge (6)
 (1) edge (2)
 (1) edge (3)
 (1) edge (6)
 (2) edge (3)
 (2) edge (6)
 (3) edge [in=284,out=204,loop] ()
 (3) edge (4)
 (3) edge (5)
 (3) edge (6)
;\end{scope}
\begin{scope}[shift={(2.800000cm,-0.000000cm)}]
\tikzset{every node/.style={circle,fill,draw=black,inner sep=0cm}}
\node[minimum size=0.04cm,label=left:0] (0-0) at (-1.050000cm,1.050000cm) {};
\node[minimum size=0.04cm,label=right:0] (0-0-outside) at (1.050000cm,-1.050000cm) {};
\node[minimum size=0.04cm] (0-1) at (-0.750000cm,0.750000cm) {};
\node[minimum size=0.04cm] (0-2) at (-0.450000cm,0.450000cm) {};
\node[minimum size=0.04cm] (0-3) at (-0.150000cm,0.150000cm) {};
\node[minimum size=0.04cm] (0-4) at (0.150000cm,-0.150000cm) {};
\node[minimum size=0.04cm] (0-5) at (0.450000cm,-0.450000cm) {};
\node[minimum size=0.04cm] (0-6) at (0.750000cm,-0.750000cm) {};
\node[minimum size=0.04cm] (1-1) at (-0.450000cm,1.050000cm) {};
\node[minimum size=0.04cm] (1-2) at (-0.150000cm,0.750000cm) {};
\node[minimum size=0.04cm] (1-3) at (0.150000cm,0.450000cm) {};
\node[minimum size=0.04cm] (1-4) at (0.450000cm,0.150000cm) {};
\node[minimum size=0.04cm] (1-5) at (0.750000cm,-0.150000cm) {};
\node[minimum size=0.04cm,label=left:1] (1-6) at (-1.050000cm,0.450000cm) {};
\node[minimum size=0.04cm,label=right:1] (1-6-outside) at (1.050000cm,-0.450000cm) {};
\node[minimum size=0.04cm] (2-2) at (0.150000cm,1.050000cm) {};
\node[minimum size=0.04cm] (2-3) at (0.450000cm,0.750000cm) {};
\node[minimum size=0.04cm] (2-4) at (0.750000cm,0.450000cm) {};
\node[minimum size=0.04cm,label=left:2] (2-5) at (-1.050000cm,-0.150000cm) {};
\node[minimum size=0.04cm,label=right:2] (2-5-outside) at (1.050000cm,0.150000cm) {};
\node[minimum size=0.04cm] (2-6) at (-0.750000cm,0.150000cm) {};
\node[minimum size=0.04cm] (3-3) at (0.750000cm,1.050000cm) {};
\node[minimum size=0.04cm,label=left:3] (3-4) at (-1.050000cm,-0.750000cm) {};
\node[minimum size=0.04cm,label=right:3] (3-4-outside) at (1.050000cm,0.750000cm) {};
\node[minimum size=0.04cm] (3-5) at (-0.750000cm,-0.450000cm) {};
\node[minimum size=0.04cm] (3-6) at (-0.450000cm,-0.150000cm) {};
\node[minimum size=0.04cm] (4-4) at (-0.750000cm,-1.050000cm) {};
\node[minimum size=0.04cm] (4-5) at (-0.450000cm,-0.750000cm) {};
\node[minimum size=0.04cm] (4-6) at (-0.150000cm,-0.450000cm) {};
\node[minimum size=0.04cm] (5-5) at (-0.150000cm,-1.050000cm) {};
\node[minimum size=0.04cm] (5-6) at (0.150000cm,-0.750000cm) {};
\node[minimum size=0.04cm] (6-6) at (0.450000cm,-1.050000cm) {};
\draw[-] (3-6) edge (0-3);
\draw[-] (0-2) edge (1-2);
\draw[-] (0-3) edge (0-4);
\draw[-] (1-2) edge (1-3);
\draw[-] (0-4) edge (0-5);
\draw[-] (1-3) edge (2-3);
\draw[-] (0-5) edge (0-6);
\draw[-] (2-3) edge (3-3);
\draw[-] (0-6) edge (1-6-outside);
\draw[-] (3-3) edge (3-4-outside);
\draw[-] (3-4) edge (3-5);
\draw[-] (1-6) edge (2-6);
\draw[-] (3-5) edge (3-6);
\draw[-] (2-6) edge (0-2);
;
\end{scope}

%% file: lattice.tex
\tikzset{every node/.style={circle,fill,draw=black,inner sep=0cm}}
\node[minimum size=0.1cm] (-2-2) at (0.000000cm,2.000000cm) {};
\node[minimum size=0.1cm] (-1-1) at (0.000000cm,1.000000cm) {};
\node[minimum size=0.1cm] (-1-2) at (0.500000cm,1.500000cm) {};
\node[minimum size=0.1cm] (-1-3) at (1.000000cm,2.000000cm) {};
\node[minimum size=0.1cm] (0-0) at (0.000000cm,0.000000cm) {};
\node[minimum size=0.1cm] (0-1) at (0.500000cm,0.500000cm) {};
\node[minimum size=0.1cm] (0-2) at (1.000000cm,1.000000cm) {};
\node[minimum size=0.1cm] (0-3) at (1.500000cm,1.500000cm) {};
\node[minimum size=0.1cm] (0-4) at (2.000000cm,2.000000cm) {};
\node[minimum size=0.04cm] (1--1) at (0.000000cm,-1.000000cm) {};
\node[minimum size=0.04cm] (1-0) at (0.500000cm,-0.500000cm) {};
\node[minimum size=0.1cm] (1-1) at (1.000000cm,0.000000cm) {};
\node[minimum size=0.1cm] (1-2) at (1.500000cm,0.500000cm) {};
\node[minimum size=0.1cm] (1-3) at (2.000000cm,1.000000cm) {};
\node[minimum size=0.1cm] (1-4) at (2.500000cm,1.500000cm) {};
\node[minimum size=0.1cm] (1-5) at (3.000000cm,2.000000cm) {};
\node[minimum size=0.04cm] (2--2) at (0.000000cm,-2.000000cm) {};
\node[minimum size=0.04cm] (2--1) at (0.500000cm,-1.500000cm) {};
\node[minimum size=0.04cm] (2-0) at (1.000000cm,-1.000000cm) {};
\node[minimum size=0.04cm] (2-1) at (1.500000cm,-0.500000cm) {};
\node[minimum size=0.1cm] (2-2) at (2.000000cm,0.000000cm) {};
\node[minimum size=0.1cm] (2-3) at (2.500000cm,0.500000cm) {};
\node[minimum size=0.1cm] (2-4) at (3.000000cm,1.000000cm) {};
\node[minimum size=0.1cm] (2-5) at (3.500000cm,1.500000cm) {};
\node[minimum size=0.1cm] (2-6) at (4.000000cm,2.000000cm) {};
\node[minimum size=0.04cm] (3--3) at (0.000000cm,-3.000000cm) {};
\node[minimum size=0.04cm] (3--2) at (0.500000cm,-2.500000cm) {};
\node[minimum size=0.04cm] (3--1) at (1.000000cm,-2.000000cm) {};
\node[minimum size=0.04cm] (3-0) at (1.500000cm,-1.500000cm) {};
\node[minimum size=0.04cm] (3-1) at (2.000000cm,-1.000000cm) {};
\node[minimum size=0.04cm] (3-2) at (2.500000cm,-0.500000cm) {};
\node[minimum size=0.1cm] (3-3) at (3.000000cm,0.000000cm) {};
\node[minimum size=0.1cm] (3-4) at (3.500000cm,0.500000cm) {};
\node[minimum size=0.1cm] (3-5) at (4.000000cm,1.000000cm) {};
\node[minimum size=0.1cm] (3-6) at (4.500000cm,1.500000cm) {};
\node[minimum size=0.1cm] (3-7) at (5.000000cm,2.000000cm) {};
\node[minimum size=0.04cm] (4--2) at (1.000000cm,-3.000000cm) {};
\node[minimum size=0.04cm] (4--1) at (1.500000cm,-2.500000cm) {};
\node[minimum size=0.04cm] (4-0) at (2.000000cm,-2.000000cm) {};
\node[minimum size=0.04cm] (4-1) at (2.500000cm,-1.500000cm) {};
\node[minimum size=0.04cm] (4-2) at (3.000000cm,-1.000000cm) {};
\node[minimum size=0.04cm] (4-3) at (3.500000cm,-0.500000cm) {};
\node[minimum size=0.1cm] (4-4) at (4.000000cm,0.000000cm) {};
\node[minimum size=0.1cm] (4-5) at (4.500000cm,0.500000cm) {};
\node[minimum size=0.1cm,label=above:$p$] (4-6) at (5.000000cm,1.000000cm) {};
\node[minimum size=0.1cm] (4-7) at (5.500000cm,1.500000cm) {};
\node[minimum size=0.1cm] (4-8) at (6.000000cm,2.000000cm) {};
\node[minimum size=0.04cm] (5--1) at (2.000000cm,-3.000000cm) {};
\node[minimum size=0.04cm] (5-0) at (2.500000cm,-2.500000cm) {};
\node[minimum size=0.04cm] (5-1) at (3.000000cm,-2.000000cm) {};
\node[minimum size=0.04cm] (5-2) at (3.500000cm,-1.500000cm) {};
\node[minimum size=0.04cm] (5-3) at (4.000000cm,-1.000000cm) {};
\node[minimum size=0.04cm] (5-4) at (4.500000cm,-0.500000cm) {};
\node[minimum size=0.1cm] (5-5) at (5.000000cm,0.000000cm) {};
\node[minimum size=0.1cm] (5-6) at (5.500000cm,0.500000cm) {};
\node[minimum size=0.1cm] (5-7) at (6.000000cm,1.000000cm) {};
\node[minimum size=0.1cm] (5-8) at (6.500000cm,1.500000cm) {};
\node[minimum size=0.1cm] (5-9) at (7.000000cm,2.000000cm) {};
\node[minimum size=0.04cm] (6-0) at (3.000000cm,-3.000000cm) {};
\node[minimum size=0.04cm] (6-1) at (3.500000cm,-2.500000cm) {};
\node[minimum size=0.04cm] (6-2) at (4.000000cm,-2.000000cm) {};
\node[minimum size=0.04cm] (6-3) at (4.500000cm,-1.500000cm) {};
\node[minimum size=0.04cm] (6-4) at (5.000000cm,-1.000000cm) {};
\node[minimum size=0.04cm] (6-5) at (5.500000cm,-0.500000cm) {};
\node[minimum size=0.1cm] (6-6) at (6.000000cm,0.000000cm) {};
\node[minimum size=0.1cm] (6-7) at (6.500000cm,0.500000cm) {};
\node[minimum size=0.1cm] (6-8) at (7.000000cm,1.000000cm) {};
\node[minimum size=0.1cm] (6-9) at (7.500000cm,1.500000cm) {};
\node[minimum size=0.1cm] (6-10) at (8.000000cm,2.000000cm) {};
\node[minimum size=0.04cm] (7-1) at (4.000000cm,-3.000000cm) {};
\node[minimum size=0.04cm] (7-2) at (4.500000cm,-2.500000cm) {};
\node[minimum size=0.04cm,label=below:$p'$] (7-3) at (5.000000cm,-2.000000cm) {};
\node[minimum size=0.04cm] (7-4) at (5.500000cm,-1.500000cm) {};
\node[minimum size=0.04cm] (7-5) at (6.000000cm,-1.000000cm) {};
\node[minimum size=0.04cm] (7-6) at (6.500000cm,-0.500000cm) {};
\node[minimum size=0.1cm] (7-7) at (7.000000cm,0.000000cm) {};
\node[minimum size=0.1cm] (7-8) at (7.500000cm,0.500000cm) {};
\node[minimum size=0.1cm] (7-9) at (8.000000cm,1.000000cm) {};
\node[minimum size=0.1cm] (7-10) at (8.500000cm,1.500000cm) {};
\node[minimum size=0.1cm,label=right:2] (7-11) at (9.000000cm,2.000000cm) {};
\node[minimum size=0.04cm] (8-2) at (5.000000cm,-3.000000cm) {};
\node[minimum size=0.04cm] (8-3) at (5.500000cm,-2.500000cm) {};
\node[minimum size=0.04cm] (8-4) at (6.000000cm,-2.000000cm) {};
\node[minimum size=0.04cm] (8-5) at (6.500000cm,-1.500000cm) {};
\node[minimum size=0.04cm] (8-6) at (7.000000cm,-1.000000cm) {};
\node[minimum size=0.04cm] (8-7) at (7.500000cm,-0.500000cm) {};
\node[minimum size=0.1cm] (8-8) at (8.000000cm,0.000000cm) {};
\node[minimum size=0.1cm] (8-9) at (8.500000cm,0.500000cm) {};
\node[minimum size=0.1cm,label=right:1] (8-10) at (9.000000cm,1.000000cm) {};
\node[minimum size=0.04cm] (9-3) at (6.000000cm,-3.000000cm) {};
\node[minimum size=0.04cm] (9-4) at (6.500000cm,-2.500000cm) {};
\node[minimum size=0.04cm] (9-5) at (7.000000cm,-2.000000cm) {};
\node[minimum size=0.04cm] (9-6) at (7.500000cm,-1.500000cm) {};
\node[minimum size=0.04cm] (9-7) at (8.000000cm,-1.000000cm) {};
\node[minimum size=0.04cm] (9-8) at (8.500000cm,-0.500000cm) {};
\node[minimum size=0.1cm,label=right:0] (9-9) at (9.000000cm,0.000000cm) {};
\node[minimum size=0.04cm] (10-4) at (7.000000cm,-3.000000cm) {};
\node[minimum size=0.04cm] (10-5) at (7.500000cm,-2.500000cm) {};
\node[minimum size=0.04cm] (10-6) at (8.000000cm,-2.000000cm) {};
\node[minimum size=0.04cm] (10-7) at (8.500000cm,-1.500000cm) {};
\node[minimum size=0.04cm,label=right:-1] (10-8) at (9.000000cm,-1.000000cm) {};
\node[minimum size=0.04cm] (11-5) at (8.000000cm,-3.000000cm) {};
\node[minimum size=0.04cm] (11-6) at (8.500000cm,-2.500000cm) {};
\node[minimum size=0.04cm,label=right:-2] (11-7) at (9.000000cm,-2.000000cm) {};
\node[minimum size=0.04cm,label=right:-3] (12-6) at (9.000000cm,-3.000000cm) {};
\path[-]
 (-1-1) edge [thick] (0-1)
 (0-1) edge [thick] (1-1)
 (1-1) edge [thick] (1-2)
 (1-2) edge [thick] (1-3)
 (1-3) edge [thick] (2-3)
 (2-3) edge [thick] (3-3)
 (3-3) edge [thick] (4-3)
 (4-3) edge [thick] (4-4)
 (4-4) edge [thick] (4-5)
 (4-5) edge [thick] (4-6)
 (4-6) edge [thick] (5-6)
 (5-6) edge [thick] (5-7)
 (5-7) edge [thick] (5-8)
 (5-8) edge [thick] (6-8)
 (6-8) edge [thick] (7-8)
 (7-8) edge [thick] (7-9)
 (7-9) edge [thick] (8-9)
 (8-9) edge [thick] (8-10)
 (4-3) edge [thick,dotted] (5-3)
 (5-3) edge [thick,dotted] (6-3)
 (6-3) edge [thick,dotted] (7-3)
 (7-3) edge [thick,dotted] (7-4)
 (7-4) edge [thick,dotted] (8-4)
 (8-4) edge [thick,dotted] (9-4)
 (9-4) edge [thick,dotted] (9-5)
 (9-5) edge [thick,dotted] (9-6)
 (9-6) edge [thick,dotted] (10-6)
 (10-6) edge [thick,dotted] (10-7)
 (10-7) edge [thick,dotted] (11-7) ;

%% file: CHM,_GenCycle.bbl
\begin{thebibliography}{10}
\bibitem{Anderson} Theodore Anderson.
\newblock Estimation of covariance matrices which are linear combinations or whose inverses
are linear combinations of given matrices. 

\newblock In R. C. Bose, I. M. Chakravati, P. C. Mahalanobis,
C. R. Rao, K. J. C. Smith (Eds.), {\em Essays in probability and statistics} (pp. 1-24). Chapel Hill: University
of North Carolina Press.


\bibitem{AmendolaEtAl}
Carlos Am{\'e}ndola, Lukas Gustafsson, Kathl{\'e}n Kohn, Orlando Marigliano,
  and Anna Seigal.
\newblock The maximum likelihood degree of linear spaces of symmetric matrices.
\newblock {\em arXiv:2012.00198}, 2020, to appear in Le Matematiche.

\bibitem{BCRV} Winfried Bruns, Aldo Conca, Claudiu Raicu, Matteo Varbaro.
\newblock Determinants, Gr{\"o}bner Bases and Cohomology.
\newblock {\em Springer Monographs in Mathematics (SMM)}, 2022.

\bibitem{Conca} Aldo Conca. 
\newblock Gr\"obner Bases of Ideals of Minors of a Symmetric Matrix. 
\newblock {\em Journal of Algebra} 166.2 (1994): 406-421.

\bibitem{CMSS} Austin Conner, Mateusz Michalek, Michael Schindler, Balazs Szendroi.
\newblock Polynomial systems admitting a simultaneous solution.
\newblock {\em arXiv:2306.02085}, 2023.


\bibitem{MRV}
Rodica Dinu, Mateusz Micha{\l}ek, and Martin Vodi\v{c}ka.
\newblock Geometry of the gaussian graphical model of the cycle.
\newblock {\em preprint}.

\bibitem{DMS2} Rodica Dinu, Mateusz Micha{\l}ek, and  Tim Seynnaeve. 
\newblock Applications of intersection theory: from maximum likelihood to chromatic polynomials. 
\newblock {\em arXiv preprint arXiv:2111.02057 (2021).}

\bibitem{DSS} Mathias Drton, Bernd Sturmfels, and Seth Sullivant. {\em Lectures on algebraic statistics}. Vol. 39. Springer Science \& Business Media, 2008.

\bibitem{GV}
Ira Gessel, G{\'e}rard Viennot.
\newblock Binomial determinants, paths, and hook length formulae. 
\newblock {\em Adv. Math.}, 58 300--321, 1985.

\bibitem{Huh} June Huh. 
\newblock Milnor numbers of projective hypersurfaces and the chromatic polynomial of graphs. 
\newblock{\em Journal of the American Mathematical Society} 25.3 (2012): 907-927.

\bibitem{Lind}
Bernt Lindstr{\"o}m.
\newblock On the vector representations of induced matroids.
\newblock{Bull.Lond.Math. Soc.}, 5 85--90, 1973.

\bibitem{MMMSV} Laurent Manivel, Mateusz Micha{\l}ek, Leonid Monin, Tim Seynnaeve, and  Martin Vodi\v{c}ka. 
\newblock Complete quadrics: Schubert calculus for gaussian models and semidefinite programming. \newblock{\em Journal of the European Mathematical Society 2023}

\bibitem{MMW} Mateusz Micha{\l}ek, Leonid Monin, and Jaros{\l}aw Wisniewski. 
\newblock Maximum Likelihood Degree, Complete Quadrics, and $\C^*$-Action. 
\newblock {\em SIAM journal on applied algebra and geometry}, 5(1), 60-85.

\bibitem{MS} Pratik Misra and Seth Sullivant. \newblock Gaussian graphical models with toric vanishing ideals. 
\newblock{\em Annals of the Institute of Statistical Mathematics} 73 (2021): 757-785.

\bibitem{StUh}
Bernd Sturmfels and Caroline Uhler.
\newblock Multivariate {G}aussian, semidefinite matrix completion, and convex algebraic geometry.
\newblock {\em Ann. Inst. Statist. Math.}, 62(4)
:603--638, 2010.

\bibitem{Seth} Seth Sullivant. 
\newblock {\em Algebraic statistics}. Vol. 194. American Mathematical Soc., 2018.

\bibitem{DST} Jan Draisma, Seth Sullivant, and Kelli Talaska. 
\newblock Trek separation for Gaussian graphical models. 
\newblock {\em The Annals of Statistics} 38.3 (2010): 1665-1685.
\end{thebibliography}
